\documentclass[a4paper, 11pt]{amsart}
\usepackage{amsthm}
\usepackage[]{amsmath}
\usepackage{amssymb}
\usepackage{enumerate}
\usepackage[]{graphicx}
\graphicspath{{./newhelm/}{./plot/}}
\usepackage{tabularx}
\usepackage[]{color}
\usepackage[left=3cm, right=3cm]{geometry}
\usepackage[colorlinks]{hyperref}
\usepackage{tikz}
\usepackage{multirow}
\usepackage{subcaption}
\usepackage{algorithm}
\usepackage{algorithmic}
\usepackage{multirow}
\usepackage{diagbox}
\usepackage{float}

\newtheorem{assumption}{Assumption}
\newtheorem{theorem}{Theorem}
\newtheorem{lemma}{Lemma}
\newtheorem{corollary}{Corollary}

\definecolor{orange}{rgb}{1, 0.5, 0}

\allowdisplaybreaks[4]

\title[Helmholtz Problem]{An Efficient Solver to Helmholtz Equations
by Recontruction Discontinuous Approximation}

\author[S.-H. Zhao]{Shuhai Zhao} \address{Institute of 
Applied Physics and Computational Mathematics, Beijing 100094, P.R. China}
\email{shuhai@pku.org.cn}

\begin{document}

\newcommand{\bm}[1]{\boldsymbol{#1}}
\newcommand{\bmr}[1]{\bm{\mr{#1}}}
\newcommand{\lj}{[ \hspace{-2pt} [}
\newcommand{\rj}{] \hspace{-2pt} ]}
\newcommand{\mb}[1]{\mathbb{#1}}
\newcommand{\mc}[1]{\mathcal{#1}}
\newcommand{\ms}[1]{\mathscr{#1}}
\newcommand{\mr}[1]{\mathrm{#1}}
\newcommand{\jump}[1]{\lj #1 \rj}
\newcommand{\aver}[1]{ \{#1\}  }
\newcommand{\wt}[1]{ \widetilde{ #1}}
\newcommand{\wh}[1]{ \widehat{ #1}}
\newcommand{\ol}[1]{ \overline{ #1}}
\newcommand{\DGnorm}[1]{ \| #1\|_{\mr{DG}}}
\newcommand{\DGtenorm}[1]{|\!|\!| #1 |\!|\!|_{\mr{DG}}}
\newcommand\anorm[1]{|\!|\!| #1 |\!|\!|_{a}}
\renewcommand{\d}[1]{\mathrm d \boldsymbol{#1}}
\newcommand{\inner}[2]{\langle #1,#2 \rangle}
\newcommand{\npar}[1]{ \frac{\partial{#1}}{\partial \un} } 
\newcommand{\du}{\mathrm{d}}
\newcommand\cnorm[1]{| #1 |_{\bmr{c}}}
\newcommand\enorm[1]{|\!|\!| #1 |\!|\!|}
\def\dim{\ifmmode \mathrm{dim} \else \text{dim}\fi}
\def\rank{\ifmmode \mathrm{rank} \else \text{rank}\fi}
\def\div{\ifmmode \mathrm{div} \else \text{div}\fi}
\def\curl{\ifmmode \mathrm{curl} \else \text{curl}\fi}
\def\Curl{\ifmmode \mathrm{Curl} \else \text{Curl}\fi}
\def\Div{\ifmmode \mathrm{Div} \else \text{Div}\fi}
\def\divt{\div_{\text{T}}}
\def\nabt{\nabla_{\text{T}}}
\def\MTh{\mc{T}_h}
\def\MT{\mc{T}}
\def\MEh{\mc{E}_h}
\def\MNh{\mc{N}_h}
\def\MThG{\mc{T}_h^{\Gamma}}
\def\MThB{\mc{T}_h^{\backslash \Gamma}}
\def\MEhG{\mc{E}_h^{\Gamma}}
\def\MEhB{\mc{E}_h^{\backslash \Gamma}}
\def\un{\bm{\mr n}}
\def\lap{\Delta}
\def\ui{\mr i}
\def\MThi{\mc{T}_{h, i}}
\def\MTho{\mc{T}_{h, 0}}
\def\MThl{\mc{T}_{h, 1}}
\def\MThic{\mc{T}_{h, i}^{\circ}}
\def\MThoc{\mc{T}_{h, 0}^{\circ}}
\def\MThlc{\mc{T}_{h, 1}^{\circ}}
\def\mR{\mc{R}}
\def\MEhI{\mc{E}_h^I}

\def\btau{\bm{\tau}}
\newcommand{\parn}[1]{\frac{\partial #1}{\partial \un}}
\def\pparn{\frac{\partial}{\partial \un}}
\newcommand{\parl}[1]{\frac{\partial #1}{\partial \bm{l}}}
\def\pparl{\frac{\partial}{\partial \bm{l}}}

\def\lLambda{\overline{\Lambda}}

\newcommand\comment[1]{}

\begin{abstract}

In this paper, an efficient solver for the Helmholtz equation
using a noval approximation space is developed.
The ingradients of the method
include the approximation space recently proposed, a discontinuous
Galerkin scheme extensively used, and a linear system solver with a
natural preconditioner. Comparing to traditional discontinuous
Galerkin methods, we refer to the new method as being more efficient
in the following sense. The
numerical performance of the new method shows that: 1) much less error
can be reached using the same degrees of freedom; 2)
the sparse matrix therein has much fewer nonzero entries so that both
the storage space and the solution time cost for the iterative solver
are reduced; 3) the preconditioner is proved to be optimal
with respect to the mesh size in the absorbing case. 
Such advantage becomes more pronounced as the approximation order increases.\\
 \textbf{keywords}: Helmholtz problem, reconstructed discontinuous approximation,
preconditioner.
\end{abstract}

\maketitle
\section{Introduction}
\label{sec_introduction}

The Helmholtz equation, derived as the frequency-domain reduction of 
the wave equation, governs time-harmonic wave propagation in heterogeneous media. 
It serves as a fundamental model across numerous scientific and
engineering disciplines, including seismic imaging, radar scattering,
acoustic design, and photonic crystal modeling
\cite{Thompson1995Galerkin, Hu2020novel,Farhat2003discontinuous, Nguyen2015hybridizable}.
The equation captures essential physical phenomena such as wave diffraction,
interference, and dispersion, making its accurate numerical solution crucial
for practical applications.

The highly oscillatory nature of Helmholtz solutions presents significant
computational challenges, particularly in high-frequency regimes. Numerical
resolution typically requires 6-10 grid points per wavelength, leading to
extremely large linear systems when modeling realistic problems. These
discretized systems yield sparse matrices that are not only large-scale but
also inherently ill-conditioned and indefinite due to the presence of both
elliptic and negative mass terms. The combination of these factors makes
the Helmholtz equation considerably more difficult to solve numerically
compared to standard elliptic partial differential equations.

The development of efficient numerical methods for Helmholtz problems has
been an active research area for decades. Early work focused on conforming
finite element methods \cite{Ihlenburg1995finite, Ihlenburg1997finite, Melenk2011wavenumber},
which provide rigorous error analysis but face challenges in handling complex
geometries and achieving high-order accuracy. More recently, discontinuous
Galerkin (DG) methods have gained popularity due to their flexibility in
mesh handling, natural accommodation of hanging nodes, and ease of
implementing high-order approximations \cite{Feng2009discontinuous,
Feng2011discontinuous, Congreve2019robust, Hoppe2013convergence}. However,
conventional DG methods typically require multiple degrees of freedom per
element, leading to increased computational costs and memory requirements.

The design of effective preconditioners for Helmholtz problems represents
another major research thrust. Traditional preconditioning strategies
often struggle with the indefinite nature of the problem. Notable approaches
include complex-shifted Laplacian preconditioners \cite{Erlangga2004, Erlangga2006, gmreshelm},
which transform the indefinite problem into a nearby definite one, and
domain decomposition methods \cite{ddm2015helm,ddm2020helm,ddm2021helm,ddm2023helm},
which exploit local solves to construct effective preconditioners. Despite
these advances, developing robust and scalable preconditioners for
Helmholtz problems remains challenging.

In this paper, we apply the reconstructed discontinuous approximation (RDA)
method \cite{Li2012efficient, Li2016discontinuous, Li2019reconstructed, Li2023preconditioned}
to the Helmholtz problem. The construction of
the finite element approximation space includes creating
an element patch for each element and solving
local least squares problems to obtain the basis functions. 
A distinctive feature of the RDA method is that
it maintains only a single degree of freedom per element, significantly
reducing the total number of degrees of freedom compared to conventional
DG methods while preserving high-order accuracy. This efficiency gain
translates to reduced memory requirements and computational costs.
We demonstrate through numerical
experiments that the RDA method achieves substantially smaller $L^2$ errors
compared to standard DG methods with equivalent numbers of degrees of freedom.
Also, we show that the resulting stiffness matrices contain significantly
fewer non-zero entries than their DG counterparts while maintaining comparable
accuracy. What's more, we develop a novel preconditioning strategy that leverages
the structure of the RDA spaces that the degrees of freedom
is independent of the reconstruction order, from which we construct a natural
preconditioner from the piecewise constant space to any high-order 
reconstructed space.
We establish convergence properties of the resulting non-Hermitian complex
linear system within the framework of standard GMRES convergence theory
and propose a specialized geometric multigrid algorithm for efficiently
solving the preconditioned system. 
The numerical experiments presented in this work systematically validate
the effectiveness of our approach. Comparative studies with conventional
DG methods demonstrate superior performance in terms of both accuracy
and computational efficiency.

The rest of this paper is organized as follows. In Section
\ref{sec_preliminaries}, we introduce the RDA finite element space
and provide the necessary background on Sobolev spaces and mesh
partitioning. Fundamental properties of the reconstruction operator
are established. In Section \ref{sec_helm_problem}, we describe the
numerical scheme for the Helmholtz problem, introduce our preconditioning
strategy, and present the main theoretical results. Section
\ref{sec_numericalresults} contains comprehensive numerical experiments
validating our theoretical findings and demonstrating the efficiency
of our method. Finally, we offer concluding remarks in Section
\ref{sec_conclusion}.


\section{Preliminaries}
\label{sec_preliminaries}

Let $\Omega \subset \mb{R}^d$ (with $d = 2, 3$) be a bounded polygonal (or polyhedral) domain with a Lipschitz boundary $\partial \Omega$. We begin by introducing some basic notation.

Let $\MTh$ be a quasi-uniform mesh partition of $\Omega$ into disjoint open triangles (or tetrahedra). Denote by $\MEh$ the set of all $(d-1)$-dimensional faces of $\MTh$, and decompose it as $\MEh = \MEh^i \cup \MEh^b$, where $\MEh^i$ and $\MEh^b$ are the sets of interior and boundary faces, respectively. Define
\begin{displaymath}
  h_K := \text{diam}(K), \quad \forall K \in \MTh, \qquad 
  h_e := \text{diam}(e), \quad \forall e \in \MEh,
\end{displaymath}
and set $h := \max_{K \in \MTh} h_K$.
The quasi-uniformity of $\MTh$ means there exists a constant $\nu > 0$ such that $h \leq \nu \min_{K \in \MTh} \rho_K$, where $\rho_K$ is the diameter of the largest ball inscribed in $K$. For a subdomain $D \subset \Omega$, we use the standard notations $L^2(D)$ and $H^s(D)$ for the complex-valued Sobolev spaces with $s \geq 0$; their corresponding semi-norms and norms are induced by the complex $L^2$ inner product.

\par The Helmholtz equation is given by
\begin{equation}
    \left\{
     \begin{aligned}    
       -\Delta u - k^2 u &= f, &&\text{in } \Omega,\\
       \frac{\partial u}{\partial \bm{n}} + \ui k u &= g,  &&\text{on } \partial \Omega,
     \end{aligned}
    \right.
    \label{eq_helm0}
\end{equation}
where $k$ is the wavenumber, $\ui = \sqrt{-1}$ is the imaginary unit, and $\bm{n}$ denotes the unit outward normal to $\Omega$. In this paper, we also consider the Helmholtz equation with absorption ($\epsilon>0$) \cite{xuejun, gmreshelm}:
\begin{equation}
    \left\{
     \begin{aligned}    
       -\Delta u - (k^2 - \ui\epsilon) u &= f, &&\text{in } \Omega,\\
       \frac{\partial u}{\partial \bm{n}} + \ui k u &= g,  &&\text{on } \partial \Omega.
     \end{aligned}
    \right.
    \label{eq_helm1}
\end{equation}
We recall the following regularity result, proved in \cite{gmreshelm}.
\begin{theorem}
    Let $\Omega \subset \mathbb{R}^d$ be a smooth domain that is star-shaped with respect to a ball. Assume the coefficient satisfies $\epsilon \lesssim k^2$, and let $f \in L^2(\Omega)$, $g \in H^{1/2}(\partial\Omega)$. Then the solution $u$ to the Helmholtz problem belongs to $H^2(\Omega)$. Moreover, there exists a constant $C$ such that
\begin{equation}
\|u\|_{H^2(\Omega)}\le C(1+k)\left(\|f\|_{L^2(\Omega)}+\|g\|_{H^{1/2}(\partial\Omega)}\right).
\end{equation}
\end{theorem}

We also introduce the trace operators that will be used in our numerical schemes. For any piecewise smooth scalar-valued function $v$ and vector-valued function $\bm{\tau}$, the jump operator $\jump{\cdot}$ and the average operator $\aver{\cdot}$ are defined on interior faces $e \in \MEh^i$ by
\begin{displaymath}
  \begin{aligned}
    \jump{v}|_e &:= v^+|_e - v^-|_e, \quad \aver{v}|_e
    := \frac{1}{2}\bigl(v^+|_e + v^-|_e\bigr), \\
    \jump{\bm{\tau}}|_e &:= \bm{\tau}^+|_e -
    \bm{\tau}^-|_e, \quad \aver{\bm{\tau}}|_e :=
    \frac{1}{2}\bigl(\bm{\tau}^+|_e + \bm{\tau}^-|_e\bigr),
  \end{aligned}
\end{displaymath}
where $v^{\pm} := v|_{K^{\pm}}$ and $\bm{\tau}^{\pm} := \bm{\tau}|_{K^{\pm}}$.
On a boundary face $e \in \MEh^b$, these operators are modified as
\begin{displaymath}
  \begin{aligned}
    \jump{v}|_e &:= v|_e, \quad \aver{v}|_e := v|_e, \\
    \jump{\bm{\tau}}|_e &:= \bm{\tau}|_e, \quad \aver{\bm{\tau}}|_e := \bm{\tau}|_e,
  \end{aligned}
\end{displaymath}
where $\bm{n}$ is the unit outward normal to $e$.
\par Throughout the paper, the letters $C$ and $C$ with subscripts denote generic constants that may change from line to line but are independent of the mesh size.

\section{Reconstructed Discontinuous Space}

This section provides a brief introduction to the RDA (Reconstructed Discontinuous Approximation) finite element space. For proofs of the lemmas and theorems, we refer the reader to \cite{Li2023preconditioned}.

\par The definition of the RDA space relies on a carefully designed linear reconstruction operator that maintains high-order approximation properties while preserving the original degrees of freedom. The reconstruction procedure involves local polynomial fitting on a neighborhood patch for each element. The key steps are as follows.

For any element $K \in \MTh$, we construct an element patch $S(K)$ consisting of $K$ itself and some surrounding elements. We use a recursive algorithm: for an element $K$, define $S_t(K)$ for $t=0,1,\dots$ by
\begin{equation}
  S_0(K) = \{K\}, \quad S_t(K) = \bigcup_{\widetilde{K} \in \MTh,\ \partial\widetilde{K} \cap \partial\widehat{K} \in \MEh,\ \widehat{K} \in S_{t-1}(K)} \widetilde{K}, \quad t \ge 1.
  \label{eq_recon_alg_patch2}
\end{equation}
The recursion stops when $t$ satisfies $\# S_t(K) \ge \# S$, and we set $S(K) := S_t(K)$. This algorithm is applied to all elements in $\MTh$ to determine the patches.

Let $I(K)$ be the set of collocation points located inside the patch $S(K)$,
\begin{displaymath}
  I(K) := \left\{ \bm{x}_{\widetilde{K}} \mid \widetilde{K} \in S(K) \right\}.
\end{displaymath}
Given a piecewise constant function $g \in U_h^0$, for every element $K \in \MTh$ we seek a polynomial $\mc{R}_K g$ of degree $m \ge 1$ defined on $S(K)$ by solving the constrained least-squares problem:
\begin{equation}
  \begin{aligned}
    \mc{R}_K g = &\mathop{\arg\min}_{p \in \mb{P}^m(S(K))} \sum_{\bm{x} \in I(K)} \bigl| p(\bm{x}) - g(\bm{x}) \bigr|^2, \\
    &\text{subject to } p(\bm{x}_K) = g(\bm{x}_K).
  \end{aligned}
  \label{eq_leastsquares}
\end{equation}
The existence and uniqueness of the solution to \eqref{eq_leastsquares} depend on the geometric distribution of the points in $I(K)$. We make the following assumption \cite{Li2016discontinuous}:
\begin{assumption}
  For any element $K \in \MTh$ and any polynomial $p \in \mb{P}^m(S(K))$,
  \begin{displaymath}
    p|_{I(K)} = 0 \quad \text{implies} \quad p|_{S(K)} \equiv 0.
  \end{displaymath}
\end{assumption}

The global linear reconstruction operator $\mc{R}$ is then defined piecewise by the local operators $\mc{R}_K$:
\begin{equation}
    \mc{R} v_h|_K := \mc{R}_K v_h|_K, \quad \forall K \in \mathcal{T}_h.
\end{equation}
Thus, $\mc{R}$ maps the piecewise constant space $U_h^0$ onto a subspace of the $m$th-order piecewise polynomial space, which we denote by $U_h^m = \mc{R} U_h^0$. The reconstructed space $U_h^m$ will serve as the approximation space in the next section.

Define the characteristic functions for the piecewise constant space $U_h^0$:
\[
e_K(\bm{x}) =
\begin{cases}
    1, & \bm{x} \in K,\\
    0, & \text{otherwise}.
\end{cases}
\]
These functions form a basis for $U_h^0$. Let $\lambda_K = \mc{R} e_K$. The following lemma ensures that $\{ \lambda_K \}$ is a basis for $U_h^m$.
\begin{lemma}
  The functions $\{ \lambda_K \}$ are linearly independent, and $U_h^m$ is spanned by $\{ \lambda_K \}$.
\end{lemma}
The action of $\mc{R}$ on a continuous function can be expressed explicitly: for any $g(\bm{x}) \in C^{m+1}(\Omega)$,
\begin{equation}
  \mc{R} g = \sum_{K \in \MTh} g(\bm{x}_K) \lambda_K(\bm{x}).
\end{equation}

Finally, we state the following approximation estimates for the reconstruction operator \cite{Li2023preconditioned}, which are essential for proving error estimates.
\begin{theorem}
  For any element $K \in \MTh$ and any $g \in H^{m+1}(\Omega)$, the following estimate holds:
  \begin{equation}
    \| g - \mc{R} g \|_{H^q(K)} \le C \Lambda_m \, h_K^{m+1-q} \, \| g \|_{H^{m+1}(S(K))}, \quad 0 \le q \le m,
    \label{eq:apperror}
  \end{equation}
  where
  \begin{displaymath}
    \Lambda(m, S(K)) := \max_{p \in \mb{P}_m(S(K))}
    \frac{\max_{\bm{x} \in S(K)} |p(\bm{x})|}{\max_{\bm{x} \in I(K)} |p(\bm{x})|}, \qquad
    \Lambda_m := \max_{K \in \MTh} \bigl( 1 + \Lambda(m, S(K)) \sqrt{\# S(K)} \bigr).
  \end{displaymath}
  \label{th_localapproximation}
\end{theorem}

\subsection{An Intuitive Analysis of Approximation Efficiency for 1D Reconstruction}
As will be demonstrated in Section \ref{sec_numericalresults}, a significant advantage of the RDA discretization is its ability to achieve smaller numerical error for the same number of degrees of freedom compared to traditional DG finite element spaces.

We first recall a standard result from interpolation theory.
\begin{lemma}
Let $f \in C^{m+1}([a,b])$, and let $x_0, x_1, \dots, x_m \in [a,b]$ be distinct nodes. Denote by $Lf(x)$ the Lagrange interpolation polynomial of degree $m$ through these nodes. Then for any $x \in [a,b]$, there exists $\xi_x \in (a,b)$ such that
\[
f(x) - Lf(x) = \frac{f^{(m+1)}(\xi_x)}{(m+1)!} \prod_{i=0}^m (x - x_i).
\]
\end{lemma}

For simplicity, we assume $g \in C^{m+1}(\bar{\Omega})$. By the interpolation formula above, there exists a point $\xi_x$ such that
\[
g(x) - \mR g(x) = \frac{g^{(m+1)}(\xi_x)}{(m+1)!} \, \omega(x),
\]
where $\omega(x) = \prod_{i=0}^m (x - x_i)$. The $L^2$ error estimate for RDA interpolation on an element $K$ is then
\[
\|g - \mR g\|_{L^2(K)} \le \frac{\|g^{(m+1)}\|_{L^\infty}}{(m+1)!} \, \|\omega\|_{L^2(K)}.
\]

Consider the case where $\# S(K) = m+1$, i.e., on a mesh with scale $h_{\text{DG}} = (m+1)h_{\text{RDA}}$, the RDA and DG methods employ the same number of degrees of freedom. Correspondingly, the error estimate for DG interpolation over the same number of DOFs is
\[
\|g - Ig\|_{L^2(S(K))} \le \frac{\|g^{(m+1)}\|_{L^\infty}}{(m+1)!} \, \|\omega\|_{L^2(S(K))}.
\]

Let $x_0$ and $x_m$ be the left and right endpoints of $S(K)$, and assume $K$ is the central element of $S(K)$. In this setting, the Lagrange interpolation polynomial of DG over $S(K)$ coincides with that of RDA over $K$. Since the interpolation error expression for RDA is identical on each element of $S(K)$ (the interpolation points are simply translated), we can compute the ratio of the $L^2$ error upper bounds for RDA and DG interpolation:
\[
C_m = \left( \frac{\int_{x_0}^{x_m} \omega(x)^2 \, dx}{(m+1) \int_{K} \omega(x)^2 \, dx} \right)^{1/2}.
\]
By a change of variable, this can be rewritten as
\[
C_m = \left( \frac{\int_0^{m+1} \widetilde{\omega}(x)^2 \, dx}{(m+1) \int_{[m/2]}^{[m/2]+1} \widetilde{\omega}(x)^2 \, dx} \right)^{1/2},
\]
where $\widetilde{\omega}(x) = \prod_{i=0}^m \left( x - \left(i+\frac{1}{2}\right) \right)$.

We perform a numerical test using the function $g(x) = \sin(20\pi x)$, with RDA mesh size $h_{\text{RDA}} = h/(m+1)$, and compute both the ratio of interpolation errors and the ratio of numerical solution errors obtained by the IPDG method.

\begin{table}[htbp]
  \centering
  \renewcommand{\arraystretch}{2.8}
  \begin{tabular}{>{\centering}p{1cm}|>{\centering}p{4cm}|>{\centering}p{3.5cm}|>{\centering\arraybackslash}p{3.5cm}}
    \hline\hline
    $m$ & Theoretical $C_m$ & $\displaystyle\frac{\|g-\mR g\|_{L^2}}{\|g- Ig\|_{L^2}}$ & FEM Error Ratio \\
    \hline
    2 & $\frac{1}{3}\sqrt{\frac{407}{263}}\approx 0.426$ & 0.417 & 0.452 \\
    \hline
    3 & $\sqrt{\frac{20219}{233579}}\approx 0.294$ & 0.295 & 0.319 \\
    \hline
    4 & $\frac{1}{13}\sqrt{\frac{128123}{36395}}\approx 0.144$ & 0.145 & 0.168 \\
    \hline
    5 & $\frac{1}{9}\sqrt{\frac{507316223}{783956623}} \approx 0.0894$ & 0.0897 & 0.116 \\
    \hline
    6 & $\sqrt{\frac{1067989595}{540028301771}}\approx 0.0445$ & 0.0446 & 0.0530 \\
    \hline\hline
  \end{tabular}
  \label{tab_ex1_results}
\end{table}


\section{Numerical Analysis}
\label{sec_helm_problem}

We begin by proposing the RDA discretization for problem \eqref{eq_helm1}. Define
\begin{displaymath}
  \begin{aligned}
    A_h(u, v) &:= \sum_{K \in \MTh} \int_{K} \nabla u \cdot \overline{\nabla v} \, \d{\bm{x}} 
   - \sum_{e \in \MEh^i} \int_{e} 
    \Bigl( \jump{u} \cdot \aver{\overline{\nabla v}} + 
    \jump{\overline{v}} \cdot \aver{\nabla u} \Bigr) \, \d{\bm{s}} \\
    &\quad + \sum_{e \in \MEh^i} \ui \int_{e}
    \mu \, \jump{u} \cdot \jump{\overline{v}}  \, \d{\bm{s}},
  \end{aligned}
\end{displaymath}
where $\mu$ is the penalty parameter:
\begin{displaymath}
  \mu = \frac{\eta}{h_e}, \quad \text{for } e \in \MEh.
\end{displaymath}
Define also
\begin{displaymath}
  a_h(u, v) := A_h(u, v) - \sum_{K \in \MTh} (k^2 - \ui\epsilon) \int_{K} u \overline{v} \, \d{\bm{x}}
    + \sum_{e \in \MEh^b} \int_{e} \ui k \, u \overline{v}  \, \d{\bm{s}}.
\end{displaymath}
The bilinear form is defined on the space $U_h := U_h^m + H^2(\Omega)$.
The numerical scheme reads: find $u_h \in U_h^m$ such that
\begin{equation}
  a_h(u_h, v_h) = l_h(v_h), \quad \forall v_h \in U_h^m,
  \label{eq_variational}
\end{equation}
where the linear form $l_h$ for $v \in U_h$ is given by
\begin{displaymath}
    l_h(v) := \sum_{K \in \MTh} \int_K f \overline{v} \, \d{\bm{x}} 
   + \sum_{e \in \MEh^b} \int_{e} g \overline{v} \, \d{\bm{s}}.
\end{displaymath}

For error analysis, we introduce the energy norms on $U_h$:
\begin{displaymath}
  \begin{aligned}
    \DGnorm{v}^2 &:= \sum_{K \in \MTh} \| \nabla v \|_{L^2(K)}^2 
                  + \sum_{e \in \MEh^i} h_e^{-1} \| \jump{v} \|_{L^2(e)}^2 
                  + \sum_{e \in \MEh^b} k \| v \|_{L^2(e)}^2,
  \end{aligned}
\end{displaymath}
and
\begin{displaymath}
  \enorm{v}^2 := \DGnorm{v}^2 + \sum_{e \in \MEh^i} h \| \aver{\nabla v} \|_{L^2(e)}^2.
\end{displaymath}
Note the following norm equivalence for $v_h \in U_h^m$:
\begin{displaymath}
  \DGnorm{v_h} \le \enorm{v_h} \le C \DGnorm{v_h}.
\end{displaymath}

\begin{lemma}
  Let $u \in H^2(\Omega)$ be the solution to \eqref{eq_helm1}, and let $u_h \in U_h^m$ be the discrete solution to \eqref{eq_variational}. Then
  \begin{equation}
    a_h(u - u_h, v_h) = 0, \quad \forall v_h \in U_h^m.
    \label{eq_orthogonality}
  \end{equation}
  \label{le_orthogonality}
\end{lemma}
\begin{proof}
  First observe that, since $u$ is continuous,
  \begin{displaymath}
    \jump{u}|_{e} = 0, \quad \forall e \in \MEh^i.
  \end{displaymath}
  Substituting $u$ into $a_h(\cdot, \cdot)$ yields
  \begin{displaymath}
    \begin{aligned}
      a_h(u, v_h) &= \sum_{K \in \MTh} \int_{K} \nabla u \cdot \overline{\nabla v_h} \, \d{\bm{x}} 
                  - \sum_{K \in \MTh} \int_{K} (k^2 - \ui\epsilon) u \overline{v_h} \, \d{\bm{x}} \\
                  &\quad - \sum_{e \in \MEh^i} \int_{e} \jump{\overline{v_h}} \cdot \aver{\nabla u} \, \d{\bm{s}}
                  + \sum_{e \in \MEh^b} \int_{e} \ui k \, u \overline{v_h} \, \d{\bm{s}}.
    \end{aligned}
  \end{displaymath}
  Multiplying \eqref{eq_helm1} by $\overline{v_h}$ and integrating by parts gives
  \begin{displaymath}
    \begin{aligned}
      \sum_{K \in \MTh} \int_K f \overline{v_h} \, \d{\bm{x}}
      &= \sum_{K \in \MTh} \int_{K} \nabla u \cdot \overline{\nabla v_h} \, \d{\bm{x}} 
         - \sum_{K \in \MTh} \int_{K} (k^2 - \ui\epsilon) u \overline{v_h} \, \d{\bm{x}} \\
        &\quad - \sum_{e \in \MEh} \int_{e} \jump{\overline{v_h}} \cdot \aver{\nabla u} \, \d{\bm{s}},
    \end{aligned}
  \end{displaymath}
  where
  \begin{displaymath}
    \sum_{e \in \MEh} \int_{e} \jump{\overline{v_h}} \cdot \aver{\nabla u} \, \d{\bm{s}}
    = \sum_{e \in \MEh^i} \int_{e} \jump{\overline{v_h}} \cdot \aver{\nabla u} \, \d{\bm{s}}
      + \sum_{e \in \MEh^b} \int_{e} (g - \ui k u) \overline{v_h} \, \d{\bm{s}}.
  \end{displaymath}
  Combining these identities, we obtain
  \begin{displaymath}
    a_h(u_h, v_h) = l_h(v_h) = a_h(u, v_h).
  \end{displaymath}
\end{proof}

\subsection{Elliptic Projection}
A key technique for error analysis is the elliptic projection. For $u, w \in H^2(\Omega)$, define $Pu$ and $Qw$ by
\begin{align*}
  A_h(Pu, v_h) + \ui k (Pu, v_h)_{L^2(\partial\Omega)} &= A_h(u, v_h) + \ui k (u, v_h)_{L^2(\partial\Omega)}, \quad \forall v_h \in U_h^m, \\
  A_h(v_h, Qw) + \ui k (v_h, Qw)_{L^2(\partial\Omega)} &= A_h(v_h, w) + \ui k (v_h, w)_{L^2(\partial\Omega)}, \quad \forall v_h \in U_h^m.
\end{align*}
Since $Pu$ and $Qw$ are solutions to elliptic problems with data $u$ and $w$, respectively, the following error estimates hold \cite{haijun2020Helm}:
\begin{lemma}
  Assuming $hk \lesssim 1$, we have
  \begin{align*}
    \enorm{u - Pu} &\le \inf_{v \in U_h^m} \enorm{u - v}, &
    \enorm{w - Qw} &\le \inf_{v \in U_h^m} \enorm{w - v}, \\
    \|u - Pu\|_{L^2(\Omega)} &\le h \inf_{v \in U_h^m} \enorm{u - v}, &
    \|w - Qw\|_{L^2(\Omega)} &\le h \inf_{v \in U_h^m} \enorm{w - v}.
  \end{align*}
  \label{le_proj}
\end{lemma}

\subsection{Error Analysis}
\begin{theorem}
  The bilinear form $a_h$ satisfies the boundedness condition:
  \begin{equation}
    |a_h(u_h, v_h)| \le C \enorm{u_h} \enorm{v_h}, \quad \forall u_h, v_h \in U_h.
  \end{equation}
\end{theorem}

\begin{theorem}
  For sufficiently large $\eta$, there holds
  \begin{equation}
    |a_h(v_h, v_h)| \ge C \enorm{v_h}^2 - k^2 \|v_h\|_{L^2(\Omega)}^2, \quad \forall v_h \in U_h^m.
    \label{coercivity}
  \end{equation}
\end{theorem}
\begin{proof}
  First note that
  \begin{equation}
    - \sum_{e \in \MEh^i} \int_{e} 2 \jump{v_h} \cdot \aver{\nabla v_h} \, \d{\bm{s}}
    \ge - \sum_{e \in \MEh^i} \frac{1}{\beta} \| h_e^{-1/2} \jump{v_h} \|^2_{L^2(e)}
        - C \beta \sum_{K \in \MTh} \| \nabla v_h \|^2_{L^2(K)},
  \end{equation}
  for any $\beta > 0$. By the Cauchy-Schwarz inequality,
  \begin{displaymath}
    - \int_{e} 2 \jump{v_h} \cdot \aver{\nabla v_h} \, \d{\bm{s}}
    \ge - \frac{1}{\beta} \| h_e^{-1/2} \jump{v_h} \|^2_{L^2(e)}
        - \beta \| h_e^{1/2} \aver{\nabla v_h} \|^2_{L^2(e)},
  \end{displaymath}
  and by the inverse estimate (with $e = \partial K^+ \cap \partial K^-$)
  \begin{displaymath}
    \| h_e^{1/2} \nabla v_h \|_{L^2(e)} \le C \| \nabla v_h \|_{L^2(K^+ \cup K^-)},
  \end{displaymath}
  we obtain
  \begin{displaymath}
    \begin{aligned}
      &\sum_{K \in \MTh} \|\nabla v_h\|_{L^2(K)}^2
      - 2 \sum_{e \in \MEh^i} \Re \int_{e} \jump{v_h} \cdot \aver{\overline{\nabla v_h}} \, \d{\bm{s}}
      + \sum_{e \in \MEh^i} \mu \|\jump{v_h}\|^2_{L^2(e)} \\
      &\ge (1 - C\beta) \sum_{K \in \MTh} \| \nabla v_h \|^2_{L^2(K)}
        + \Bigl(\eta - \frac{1}{\beta}\Bigr) \sum_{e \in \MEh} \| h_e^{-1/2} \jump{v_h} \|^2_{L^2(e)}.
    \end{aligned}
  \end{displaymath}
  Choosing $\beta = 1/(2C)$ and $\eta$ sufficiently large, and using
  \begin{displaymath}
    \begin{aligned}
      a_h(v_h, v_h) &= \sum_{K \in \MTh} \|\nabla v_h\|_{L^2(K)}^2
        - 2 \sum_{e \in \MEh^i} \Re \int_{e} \jump{v_h} \cdot \aver{\overline{\nabla v_h}} \, \d{\bm{s}}
        + \sum_{e \in \MEh^i} \ui \mu \|\jump{v_h}\|^2_{L^2(e)} \\
        &\quad - \sum_{K \in \MTh} (k^2 - \ui\epsilon) \| v_h \|^2_{L^2(K)}
        + \sum_{e \in \MEh^b} \ui k \|v_h\|^2_{L^2(e)},
    \end{aligned}
  \end{displaymath}
  we conclude
  \begin{displaymath}
    \DGnorm{v_h}^2 \lesssim |a_h(v_h, v_h) + 2k^2 (v_h, v_h)|.
  \end{displaymath}
\end{proof}

The following lemma provides an interpolation error estimate in the $\enorm{\cdot}$ norm.
\begin{lemma}
  There exists a constant $C$ such that
  \begin{equation}
    \enorm{v - \mc{R} v} \le C \Lambda_m (1 + k^2 h^2)^{1/2} h^{m} \| v \|_{H^{m+1}(\Omega)},
    \quad \forall v \in H^{m+1}(\Omega).
    \label{le_ch5}
  \end{equation}
\end{lemma}
\begin{proof}
  By Theorem \ref{th_localapproximation},
  \begin{displaymath}
    \sum_{K \in \MTh} \| \nabla v - \nabla(\mc{R} v) \|^2_{L^2(K)}
    \le C \Lambda_m^2 h^{2m} \| v \|^2_{H^{m+1}(\Omega)}.
  \end{displaymath}
  Similarly,
  \begin{displaymath}
    \sum_{K \in \MTh} k^2 \| v - \mc{R} v \|^2_{L^2(K)}
    \le C \Lambda_m^2 k^2 h^{2m+2} \| v \|^2_{H^{m+1}(\Omega)},
  \end{displaymath}
  and by trace and inverse inequalities,
  \begin{displaymath}
    \sum_{e \in \MEh^b} k \| v - \mc{R} v \|^2_{L^2(e)}
    \le C \Lambda_m^2 k h^{2m+1} \| v \|^2_{H^{m+1}(\Omega)}.
  \end{displaymath}
  The remaining terms are estimated analogously.
\end{proof}

We introduce the adjoint problem
\begin{equation}
  \left\{
    \begin{aligned}
      -\Delta w - (k^2 - \ui\epsilon) w &= u - u_h, && \text{in } \Omega, \\
      \frac{\partial w}{\partial \bm{n}} + \ui k w &= 0, && \text{on } \partial \Omega.
    \end{aligned}
  \right.
  \label{eq_H3}
\end{equation}
For the error analysis, we assume the following regularity for $w$:
\begin{equation}
  \|w\|_{H^2(\Omega)} \le (1 + k) \|u - u_h\|_{L^2(\Omega)}.
\end{equation}

\begin{theorem}
  Assume the solution $u$ to \eqref{eq_helm1} satisfies $u \in H^{m+1}(\Omega)$. For a sufficiently large penalty $\eta$ and under the condition $k^3 h^2 \lesssim 1$, there exists a constant $C$ such that
  \begin{equation}
    \enorm{u - u_h} \le C \Lambda_m (1 + k^2 h^2)^{1/2} h^{m} \| u \|_{H^{m+1}(\Omega)}.
    \label{errorestimate0}
  \end{equation}
  \label{L2estimate}
\end{theorem}
\begin{proof}
  Taking the inner product of \eqref{eq_H3} with $u - u_h$ and conjugating gives
  \begin{equation}
    \begin{aligned}
      \|u - u_h\|_{L^2(\Omega)}^2
      &= A_h(u - u_h, w) - (k^2 - \ui\epsilon)(u - u_h, w)_{L^2(\Omega)} + \ui k (u - u_h, w)_{L^2(\partial\Omega)} \\
      &= A_h(u - u_h, w - Qw) - (k^2 - \ui\epsilon)(u - u_h, w - Qw)_{L^2(\Omega)} \\
      &\quad + \ui k (u - u_h, w - Qw)_{L^2(\partial\Omega)} \\
      &= A_h(u - Pu, w - Qw) - (k^2 - \ui\epsilon)(u - u_h, w - Qw)_{L^2(\Omega)} \\
      &\quad + \ui k (u - Pu, w - Qw)_{L^2(\partial\Omega)} \\
      &\le \enorm{u - Pu} \, \enorm{w - Qw} + k^2 \|u - u_h\|_{L^2(\Omega)} \|w - Qw\|_{L^2(\Omega)}.
    \end{aligned}
    \label{estimate}
  \end{equation}
  Using the properties of the elliptic projection,
  \begin{align*}
    \enorm{w - Qw} &\le h (1 + k^2 h^2)^{1/2} \Lambda_m \|w\|_{H^2(\Omega)}
                    \le k h (1 + k^2 h^2)^{1/2} \Lambda_m \|u - u_h\|_{L^2(\Omega)}, \\
    \|w - Qw\|_{L^2(\Omega)} &\le C k h^2 (1 + k^2 h^2)^{1/2} \Lambda_m \|u - u_h\|_{L^2(\Omega)}.
  \end{align*}
  Therefore,
  \begin{align*}
    \|u - u_h\|_{L^2(\Omega)}^2
    &\le \enorm{u - Pu} \, k h (1 + k^2 h^2)^{1/2} \Lambda_m \|u - u_h\|_{L^2(\Omega)} \\
    &\quad + k^3 h^2 (1 + k^2 h^2)^{1/2} \Lambda_m^2 \|u - u_h\|_{L^2(\Omega)}^2.
  \end{align*}
  If $k^3 h^2 \Lambda_m^2 < 1/2$, then
  \begin{displaymath}
    \|u - u_h\|_{L^2(\Omega)} \le C k h (1 + k^2 h^2)^{1/2} \Lambda_m \enorm{u - Pu}.
  \end{displaymath}
  The result follows from Lemma \ref{le_proj}.
\end{proof}

\begin{corollary}
  Under the same assumptions as Theorem \ref{L2estimate}, there exists a constant $C$ such that
  \begin{equation}
    \|u - u_h\|_{L^2(\Omega)} \le C k \Lambda_m (1 + k^2 h^2) h^{m+1} \| u \|_{H^{m+1}(\Omega)}.
    \label{errorestimate}
  \end{equation}
\end{corollary}
\begin{proof}
  Taking $v_h = \mR u - u_h$ and using Theorem \ref{coercivity}, we obtain
  \begin{displaymath}
    \begin{aligned}
      C \enorm{\mR u - u_h}^2
      &\le \DGnorm{\mR u - u_h}^2 \\
      &\le |a_h(\mR u - u_h, \mR u - u_h) + k^2 (\mR u - u_h, \mR u - u_h)| \\
      &= |a_h(\mR u - u, \mR u - u_h) + k^2 (\mR u - u_h, \mR u - u_h)| \\
      &\le |a_h(\mR u - u, \mR u - u_h)| + k^2 |(\mR u - u, \mR u - u_h)| \\
      &\quad + k^2 |(u - u_h, \mR u - u_h)|.
    \end{aligned}
  \end{displaymath}
  Hence,
  \begin{displaymath}
    \enorm{\mR u - u_h}^2
    \le C_0 \enorm{\mR u - u} \, \enorm{\mR u - u_h}
      + C_0 k \|u - u_h\|_{L^2(\Omega)} \, \enorm{\mR u - u_h}.
  \end{displaymath}
  Applying \eqref{L2estimate} completes the proof.
\end{proof}

\subsection{Preconditioner}
We now turn our attention to preconditioning. The matrix form of the system is
\begin{displaymath}
  a_h(u_h, v_h) = l_h(v_h) \quad \Rightarrow \quad A_{\epsilon} \bm{x} = \bm{b},
\end{displaymath}
where $A_{\epsilon} \in \mb{C}^{n_e \times n_e}$, $\bm{x}, \bm{b} \in \mb{C}^{n_e}$. Note that the matrix size is independent of the order $m$.

We choose a preconditioner based on the lowest-order discretization: $P$ corresponds to the bilinear form $a_h^0(\cdot, \cdot)$ acting on $U_h^0 \times U_h^0$:
\begin{displaymath}
  a_h^0(u_h, v_h) = \sum_{e \in \MEh^i} \int_{e} \eta h^{-1} \jump{u_h} \cdot \jump{\overline{v_h}} \, \d{s}
    + \sum_{K \in \MTh} \int_K k^2 u_h \overline{v_h} \, \d{x}
    + \sum_{e \in \MEh^b} \int_{e} k u_h \overline{v_h} \, \d{s}.
\end{displaymath}

We consider Krylov subspace iterative methods for solving the preconditioned linear system. Since the system is non-positive-definite, we employ the preconditioned GMRES (PGMRES) method. We first establish the following lemma for the absorbing case ($\epsilon > 0$).

\begin{lemma}
  If $0 < \epsilon \lesssim k^2$, there exists $\alpha > 0$ such that
  \begin{displaymath}
    |a_h(\mc{R} v_h, \mc{R} v_h)| \ge \alpha \frac{\epsilon}{k^2} \DGnorm{\mc{R} v_h}^2, \quad \forall v_h \in U_h^0.
    \label{imag}
  \end{displaymath}
\end{lemma}
\begin{proof}
  Let $p, q$ be such that
  \begin{displaymath}
    k^2 - \ui \epsilon = (p - \ui q)^2,
  \end{displaymath}
  which implies $p \sim k$, $q \sim \epsilon/k$, since $\epsilon \lesssim k^2$.
  We have
  \begin{displaymath}
    \begin{aligned}
      a_h(\mc{R} v_h, \mc{R} v_h)
      &= \sum_{K \in \MTh} \|\nabla \mc{R} v_h\|_{L^2(K)}^2
         - 2 \sum_{e \in \MEh^i} \Re \int_{e} \jump{\mc{R} v_h} \cdot \aver{\overline{\nabla \mc{R} v_h}} \, \d{\bm{s}} \\
        &\quad + \sum_{e \in \MEh^i} \ui \mu \|\jump{\mc{R} v_h}\|^2_{L^2(e)}
         - \sum_{K \in \MTh} (k^2 - \ui\epsilon) \| \mc{R} v_h \|^2_{L^2(K)} \\
        &\quad + \sum_{e \in \MEh^b} \ui k \|\mc{R} v_h\|^2_{L^2(e)}.
    \end{aligned}
  \end{displaymath}
  Then,
  \begin{displaymath}
    \begin{aligned}
      &\Im \bigl( (p + \ui q) a_h(\mc{R} v_h, \mc{R} v_h) \bigr) \\
      &= q \Bigl( \sum_{K \in \MTh} \|\nabla \mc{R} v_h\|_{L^2(K)}^2
         - 2 \sum_{e \in \MEh^i} \Re \int_{e} \jump{\mc{R} v_h} \cdot \aver{\overline{\nabla \mc{R} v_h}} \, \d{\bm{s}} \Bigr) \\
        &\quad + \sum_{K \in \MTh} q (p^2 + q^2) \| \mc{R} v_h \|^2_{L^2(K)}
         + \sum_{e \in \MEh^b} p k \|\mc{R} v_h\|^2_{L^2(e)}
         + \sum_{e \in \MEh^i} \mu p \|\jump{\mc{R} v_h}\|^2_{L^2(e)}.
    \end{aligned}
  \end{displaymath}
  Consequently,
  \begin{displaymath}
    \begin{aligned}
      |a_h(\mc{R} v_h, \mc{R} v_h)|
      &\ge \frac{C q}{\sqrt{p^2 + q^2}} \Bigl( \sum_{K \in \MTh} \|\nabla \mc{R} v_h\|_{L^2(K)}^2
         + \sum_{e \in \MEh^i} \mu \|\jump{\mc{R} v_h}\|^2_{L^2(e)} \Bigr) \\
        &\quad + \sum_{K \in \MTh} q \sqrt{p^2 + q^2} \| \mc{R} v_h \|^2_{L^2(K)}
         + \sum_{e \in \MEh^b} \frac{p k}{\sqrt{p^2 + q^2}} \|\mc{R} v_h\|^2_{L^2(e)} \\
      &\ge \alpha \frac{\epsilon}{k^2} \DGnorm{\mc{R} v_h}^2.
    \end{aligned}
  \end{displaymath}
\end{proof}

The following norm equivalence is crucial for the convergence of PGMRES.
\begin{lemma}
  Assume $h k \lesssim 1$. Then there exists a constant $C$ independent of the mesh size such that
  \begin{equation}
    \DGnorm{v_h} \le C \DGnorm{\mc{R} v_h} \le C \Lambda_m \DGnorm{v_h}, \quad \forall v_h \in U_h^0.
    \label{normequi}
  \end{equation}
  \label{lemma4}
\end{lemma}
\begin{proof}
  We first prove the lower bound in \eqref{normequi}. For the volume term, using the approximation property,
  \begin{equation}
    \|v_h - \mR v_h\|_{L^2(K)} \le C h \| \nabla \mR v_h \|_{L^2(K)},
  \end{equation}
  hence
  \begin{displaymath}
    k^2 \| v_h \|_{L^2(K)}^2 \le k^2 \|\mR v_h \|_{L^2(K)}^2 + h^2 k^2 \| \nabla \mR v_h \|_{L^2(K)}^2.
  \end{displaymath}
  For an interior face $e \in \MEh^i$ with $e = \partial K_+ \cap \partial K_-$, let $v_{K^\pm} := v_h|_{K^\pm}$. By the constraint in \eqref{eq_leastsquares}, $v_{K^\pm} = I_{K^\pm} (\mc{R}_{K^\pm} v_h)$, where $I_K$ denotes the constant interpolation at $\bm{x}_K$. Using inverse estimates,
  \begin{displaymath}
    \begin{aligned}
      h_e^{-1} \| \jump{v_h} \|^2_{L^2(e)}
      &\le C h^{-1} \Bigl( \| \mc{R}_{K^+} v_h - I_{K^+} \mc{R}_{K^+} v_h \|_{L^2(e)}^2 \\
      &\quad + \| \mc{R}_{K^-} v_h - I_{K^-} \mc{R}_{K^-} v_h \|_{L^2(e)}^2
         + \| \mc{R}_{K^+} v_h - \mc{R}_{K^-} v_h \|_{L^2(e)}^2 \Bigr) \\
      &\le C h^{-2} \Bigl( \| \mc{R}_{K^+} v_h - I_{K^+} \mc{R}_{K^+} v_h \|_{L^2(K^+)}^2 \\
      &\quad + \| \mc{R}_{K^-} v_h - I_{K^-} \mc{R}_{K^-} v_h \|_{L^2(K^-)}^2
         + \| \mc{R}_{K^+} v_h - \mc{R}_{K^-} v_h \|_{L^2(e)}^2 \Bigr) \\
      &\le C \Bigl( \| \nabla \mc{R}_{K^+} v_h \|_{L^2(K^+)}^2 + \| \nabla \mc{R}_{K^-} v_h \|_{L^2(K^-)}^2
         + h_e^{-1} \| \jump{\mc{R} v_h} \|_{L^2(e)}^2 \Bigr).
    \end{aligned}
  \end{displaymath}
  For a boundary face $e \in \MEh^b$ with $e \subset \partial K$, a similar estimate yields
  \begin{align*}
    h_e^{-1} \| v_h \|_{L^2(e)}^2
    &\le C \Bigl( \| \nabla \mc{R}_{K} v_h \|_{L^2(K)}^2 + h_e^{-1} \| \jump{\mc{R}_{K} v_h} \|_{L^2(e)}^2 \Bigr), \\
    k \| v_h \|_{L^2(e)}^2
    &\le C \Bigl( k h \| \nabla \mc{R}_{K} v_h \|_{L^2(K)}^2 + k \| \jump{\mc{R}_{K} v_h} \|_{L^2(e)}^2 \Bigr).
  \end{align*}
  Summing over all faces and using $k h \lesssim 1$, we obtain $\DGnorm{v_h} \le C \DGnorm{\mc{R} v_h}$.

  For the upper bound, note that
  \begin{displaymath}
    k^2 \|\mR v_h\|_{L^2(K)}^2 \le k^2 \| v_h \|_{L^2(K)}^2 + h^2 k^2 \| \nabla \mR v_h \|_{L^2(K)}^2.
  \end{displaymath}
  For interior faces, using the same notation,
  \begin{displaymath}
    \begin{aligned}
      h_e^{-1} \| \jump{\mc{R} v_h} \|^2_{L^2(e)}
      &\le C h_e^{-1} \Bigl( \| \mc{R}_{K^+} v_h - v_{K^+} \|_{L^2(e)}^2
         + \| \mc{R}_{K^-} v_h - v_{K^-} \|_{L^2(e)}^2
         + \| \jump{v_h} \|_{L^2(e)}^2 \Bigr) \\
      &\le C h_e^{-2} \Bigl( \| \mc{R}_{K^+} v_h - v_{K^+} \|_{L^2(K^+)}^2
         + \| \mc{R}_{K^-} v_h - v_{K^-} \|_{L^2(K^-)}^2 \Bigr)
         + h_e^{-1} \| \jump{v_h} \|_{L^2(e)}^2 \\
      &\le C \Bigl( \| \nabla \mc{R}_{K^+} v_h \|_{L^2(K^+)}^2 + \| \nabla \mc{R}_{K^-} v_h \|_{L^2(K^-)}^2
         + h_e^{-1} \| \jump{v_h} \|_{L^2(e)}^2 \Bigr).
    \end{aligned}
  \end{displaymath}
  For boundary faces,
  \begin{displaymath}
    k \| \jump{\mc{R} v_h} \|_{L^2(e)}^2
    \le C \Bigl( k h \| \nabla \mc{R}_{K} v_h \|_{L^2(K)}^2 + k \| \jump{v_h} \|_{L^2(e)}^2 \Bigr).
  \end{displaymath}
  Combining these estimates and using $k h \lesssim 1$,
  \begin{displaymath}
    \DGnorm{\mc{R} v_h}^2 \le \sum_{K \in \MTh} \| \nabla \mc{R} v_h \|_{L^2(K)}^2 + C \DGnorm{v_h}^2.
  \end{displaymath}
  For each $K \in \MTh$, let $v_h$ attain its maximum and minimum on $S(K)$ at $K'$ and $K''$, respectively. As shown in \cite{Li2023preconditioned}, for a sequence of neighboring elements $\wt{K}_0 = K', \wt{K}_1, \ldots, \wt{K}_M = K''$ in $S(K)$,
  \begin{align}
    \| \nabla \mc{R} v_h \|_{L^2(K)}^2
    &\le C h^{d-2} \Lambda_m^2 \sum_{j=0}^{M-1} (v_h|_{\wt{K}_j} - v_h|_{\wt{K}_{j+1}})^2,
  \end{align}
  implying
  \begin{equation}
    \sum_{K \in \MTh} \| \nabla \mc{R} v_h \|_{L^2(K)}^2
    \le C \sum_{e \in \MEh^i} \Lambda_m^2 h_e^{-1} \| \jump{v_h} \|_{L^2(e)}^2.
  \end{equation}
  Combining these bounds yields $\DGnorm{\mc{R} v_h} \le C \Lambda_m \DGnorm{v_h}$.
\end{proof}

The convergence of the iterative method is established via the Elman estimate.
\begin{theorem}
  Under the conditions of Lemma \ref{imag} and Lemma \ref{lemma4},
  \begin{displaymath}
    \begin{aligned}
      \| P^{-1} A_{\epsilon} \bm{v} \|_{P} &\le C_1 \Lambda_m^2 \| \bm{v} \|_{P}, & \forall \bm{v} \in \mb{C}^{n_e}, \\
      |(P^{-1} A_{\epsilon} \bm{v}, \bm{v})_{P}| &\ge C_2 \frac{\epsilon}{k^2} (\bm{v}, \bm{v})_{P}, & \forall \bm{v} \in \mb{C}^{n_e}.
    \end{aligned}
  \end{displaymath}
\end{theorem}
\begin{proof}
  From Lemma \ref{lemma4},
  \begin{displaymath}
    \begin{aligned}
      \| P^{-1} A_{\epsilon} \bm{v} \|_{P}^2
      &= \bm{v}^* (A_{\epsilon})^* P^{-1} A_{\epsilon} \bm{v} \\
      &\le \bigl( \sigma_{\max}(P^{-1/2} A_{\epsilon} P^{-1/2}) \bigr)^2 \bm{w}^* \bm{w} \quad (\bm{w} = P^{1/2} \bm{v}) \\
      &\le C \Lambda_m^4 \bm{v}^* P \bm{v} = C \Lambda_m^4 \| \bm{v} \|_{P}^2.
    \end{aligned}
  \end{displaymath}
  Moreover,
  \begin{displaymath}
    |(P^{-1} A_{\epsilon} \bm{v}, \bm{v})_{P}| = |\bm{v}^* A_{\epsilon} \bm{v}|
    \ge \frac{\epsilon}{k^2} \enorm{\mc{R} v}^2
    \ge C \frac{\epsilon}{k^2} \| \bm{v} \|_{P}^2.
  \end{displaymath}
  This completes the proof.
\end{proof}

As an illustration, Figures \ref{tzz1} and \ref{tzz2} display the eigenvalues of the original and preconditioned systems on the complex plane for third-order reconstruction, with $k=10$, for both $\epsilon=0$ and $\epsilon=k^2$.

\begin{figure}[htbp]
  \centering
  \includegraphics[width=0.4\textwidth]{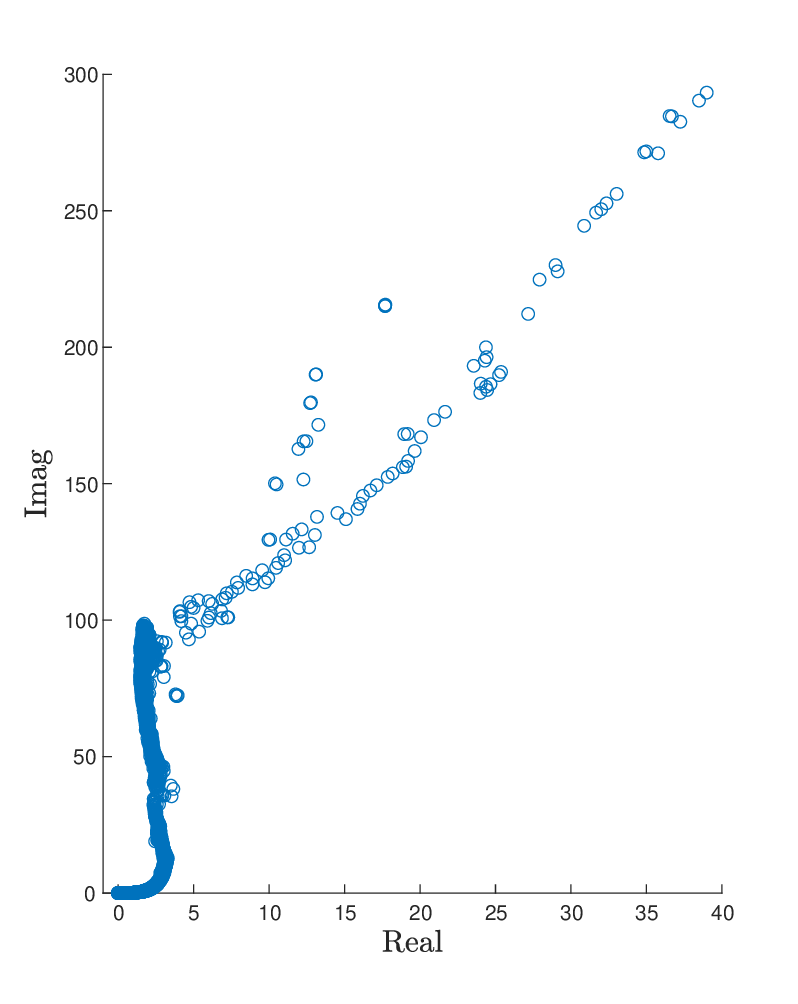}
  \hspace{30pt}
  \includegraphics[width=0.4\textwidth]{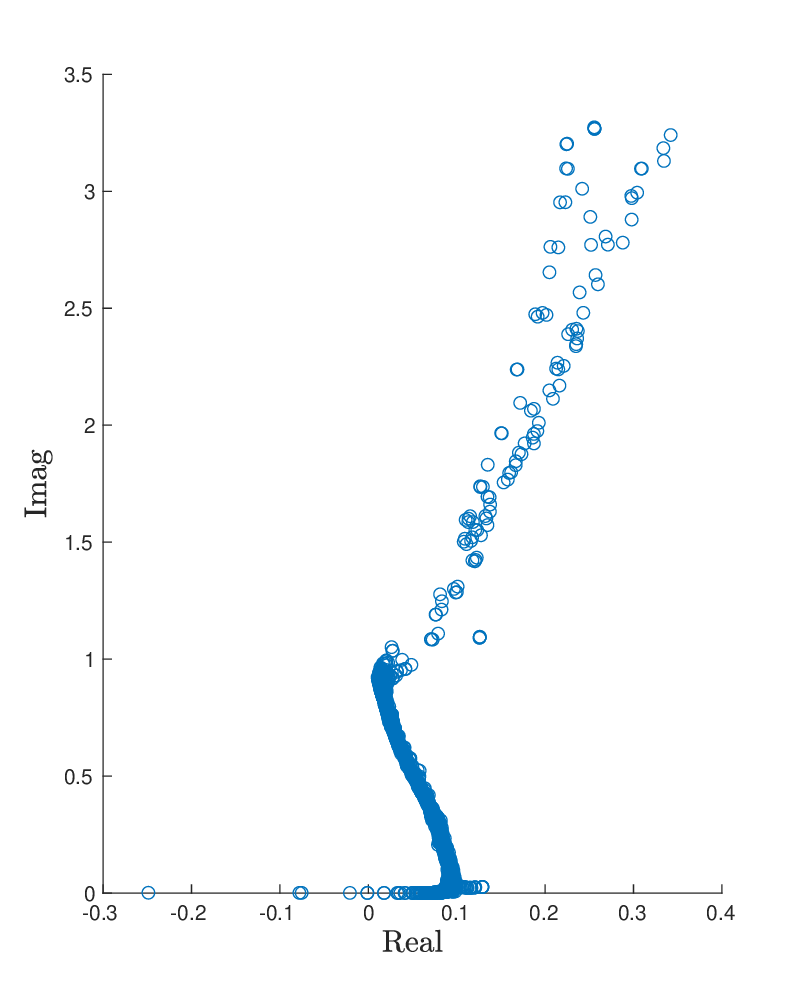}
  \caption{Eigenvalues for $\epsilon=0$. Left: $A_{\epsilon}$, right: $P^{-1}A_{\epsilon}$.}
  \label{tzz1}
\end{figure}

\begin{figure}[htbp]
  \centering
  \includegraphics[width=0.4\textwidth]{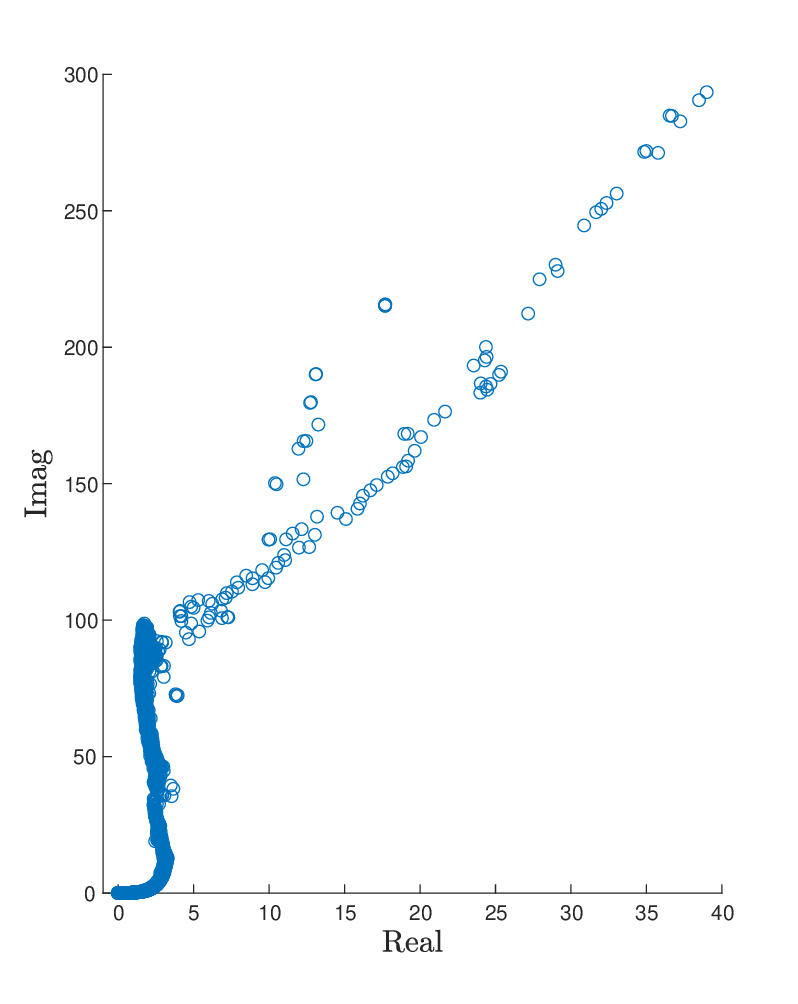}
  \hspace{30pt}
  \includegraphics[width=0.4\textwidth]{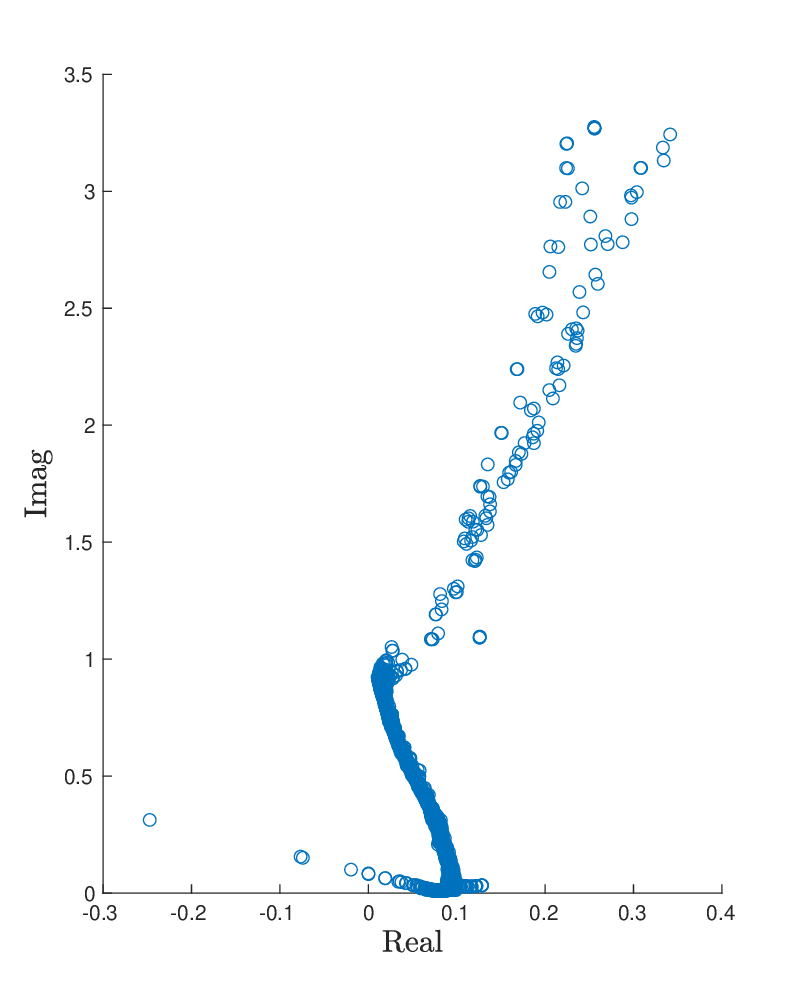}
  \caption{Eigenvalues for $\epsilon=k^2$. Left: $A_{\epsilon}$, right: $P^{-1}A_{\epsilon}$.}
  \label{tzz2}
\end{figure}

\par In each Krylov iteration, we need to compute $P^{-1} \bm{x}$, i.e., solve a linear system $P \bm{y} = \bm{z}$. We propose a geometric multigrid method for this purpose. Given a sequence of nested meshes $\mc{T}_1, \mc{T}_2, \ldots, \mc{T}_r$, let $U_k^0$ be the piecewise constant space on $\mc{T}_k$, and $P_k$ the discretization of $a_h^0(\cdot, \cdot)$ on $U_k^0 \times U_k^0$. Then
\begin{displaymath}
  U_1^0 \subset U_2^0 \subset U_3^0 \subset \cdots \subset U_r^0.
\end{displaymath}
Define the prolongation and restriction operators as
\begin{displaymath}
  \begin{aligned}
    I_{k}^{k+1} &: U_{k}^0 \rightarrow U_{k+1}^0, \quad I_{k}^{k+1} v_h = v_h, \\
    I_{k+1}^k   &: U_{k+1}^0 \rightarrow U_{k}^0, \quad I_{k+1}^k = (I_{k}^{k+1})^T.
  \end{aligned}
\end{displaymath}
Numerical experiments confirm that this geometric multigrid solver performs well for the system $P \bm{y} = \bm{z}$.


\section{Numerical Results}
\label{sec_numericalresults}
In this section, we perform numerical experiments
 to test the performance of the proposed method.
We first examine the high-order convergence of the RDA space $U_h^m$, 
and then evaluate the performance of the preconditioner.

\begin{figure}
  \centering
  \includegraphics[width=0.4\textwidth]{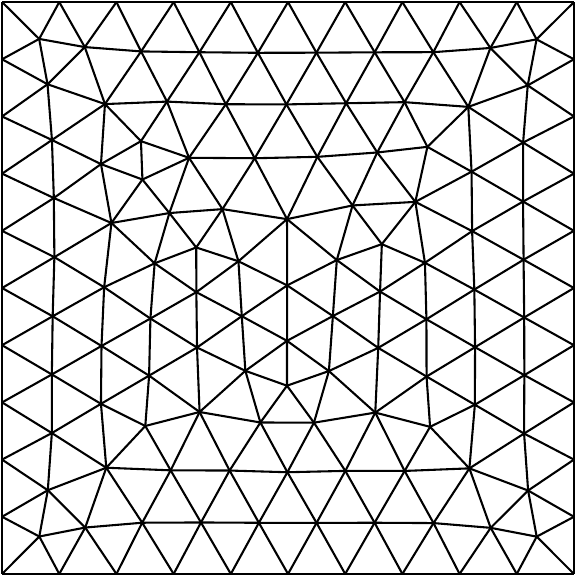}
  \hspace{25pt}
  \includegraphics[width=0.4\textwidth]{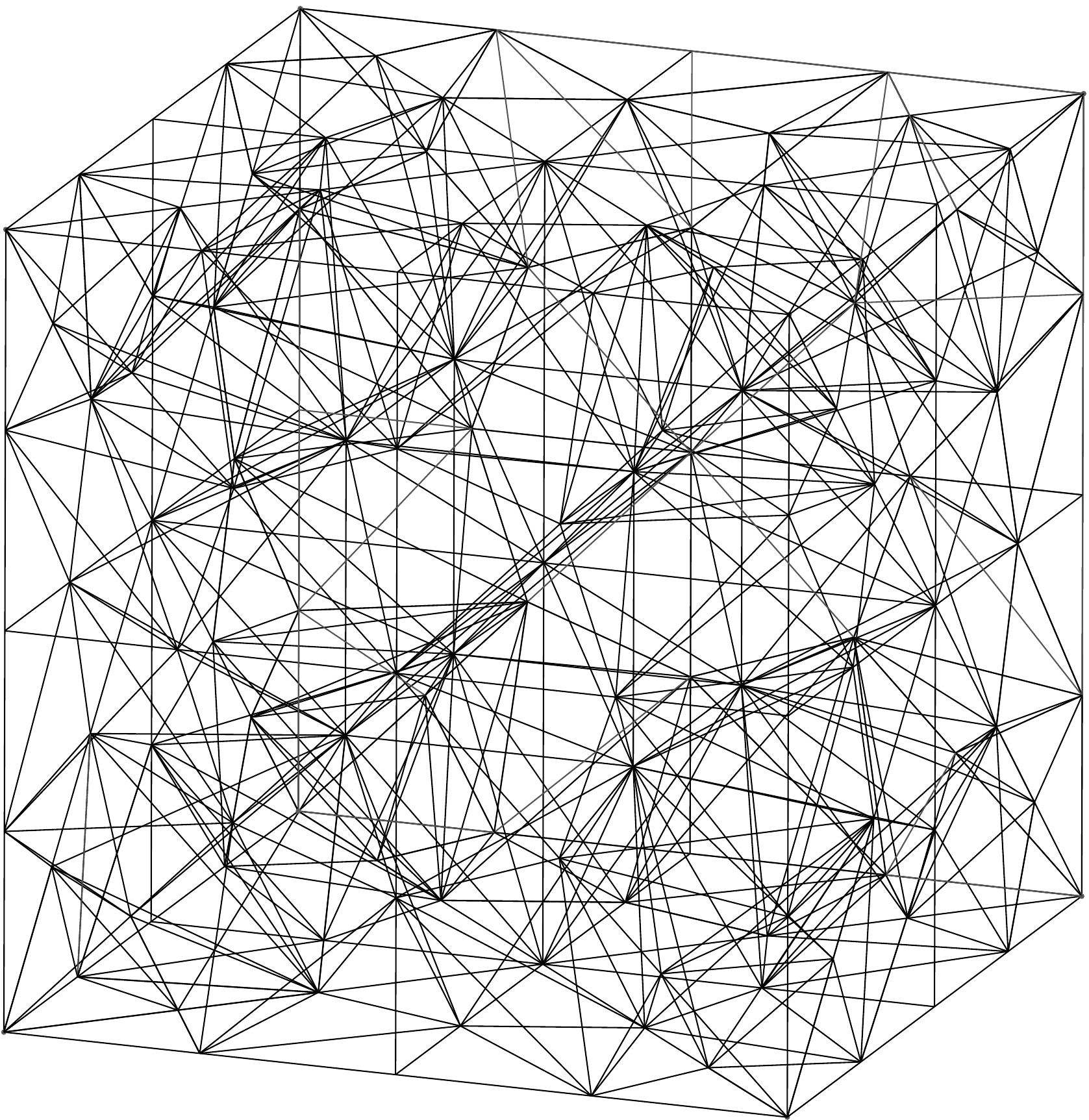}
  \caption{2D triangular partition with $ h = 1/10$ (left) and 3D tetrahedral partition with $h = 1/4$ (right).}
  \label{fig_partition}
\end{figure}

\begin{table}[htp]
\begin{minipage}[t]{0.3\textwidth}
  \centering
  \begin{tabular}{p{0.6cm}|p{0.6cm}|p{0.6cm}|p{0.6cm}|p{0.6cm}|p{0.6cm}}
  \hline\hline
  $m$    & 2 & 3 & 4 & 5 & 6  \\ \hline
  $\# S$ & 9 & 16 & 21 & 29 & 38 \\ 
  \hline\hline
  \end{tabular}
  \vspace{0.5cm}
  \label{tab_patch_2d}
\end{minipage}
\hspace{2cm}
\begin{minipage}[t]{0.3\textwidth}
  \centering
  \begin{tabular}{p{0.6cm}|p{0.6cm}|p{0.6cm}|p{0.6cm}}
    \hline\hline
   $m$    & 2 & 3 & 4  \\ \hline
   $\# S$ & 13 & 28 & 40  \\ 
    \hline\hline
  \end{tabular}
  \vspace{0.5cm}
  \label{tab_patch_3d}
\end{minipage}
\caption{The $\# S$ used in 2D and 3D examples.} 
\end{table}

\noindent \textbf{Example 1}
We first solve the pure Helmholtz problem on the square domain 
$\Omega = (0,1)^2$, choosing the exact solution as
\begin{displaymath}
  u(x,y) = e^{i k (x \cos{\frac{\pi}{5}} + y \sin{\frac{\pi}{5}})}.
\end{displaymath}
We test the convergence under the norms $\DGnorm{\cdot}$ and 
$\| \cdot \|_{L^2(\Omega)}$
on a series of uniformly refined meshes with $h = 1/10, 1/20, 
\ldots, 1/160$. We consider wavenumbers $k=5, 10, 20$ and plot the 
log-log errors in Figure~\ref{ch5_1}, Figure~\ref{ch5_2}, and Figure~\ref{ch5_3}.
In all experiments, we observe optimal convergence rates for the errors under both norms,
which agrees with our theoretical analysis.

\begin{figure}[htbp]
  \centering
  \includegraphics[width=0.45\textwidth]{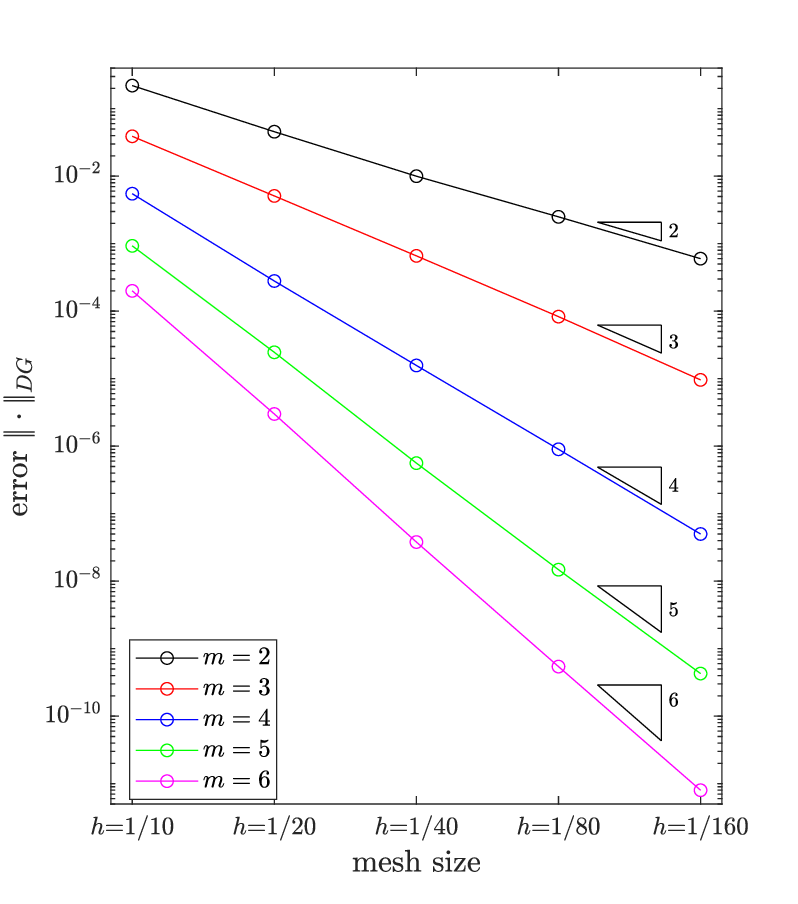}
  \hspace{30pt}
  \includegraphics[width=0.45\textwidth]{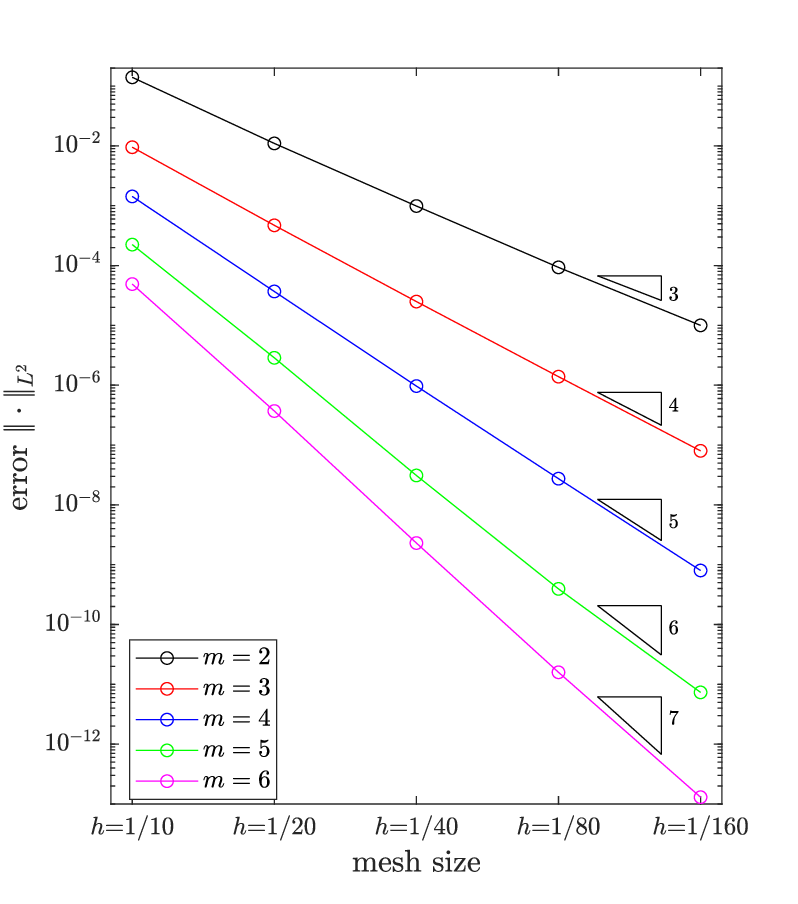}
    \caption{2D accuracy test, $k=5$.}
    \label{ch5_1}
\end{figure}

\begin{figure}[htbp]
  \centering
  \includegraphics[width=0.45\textwidth]{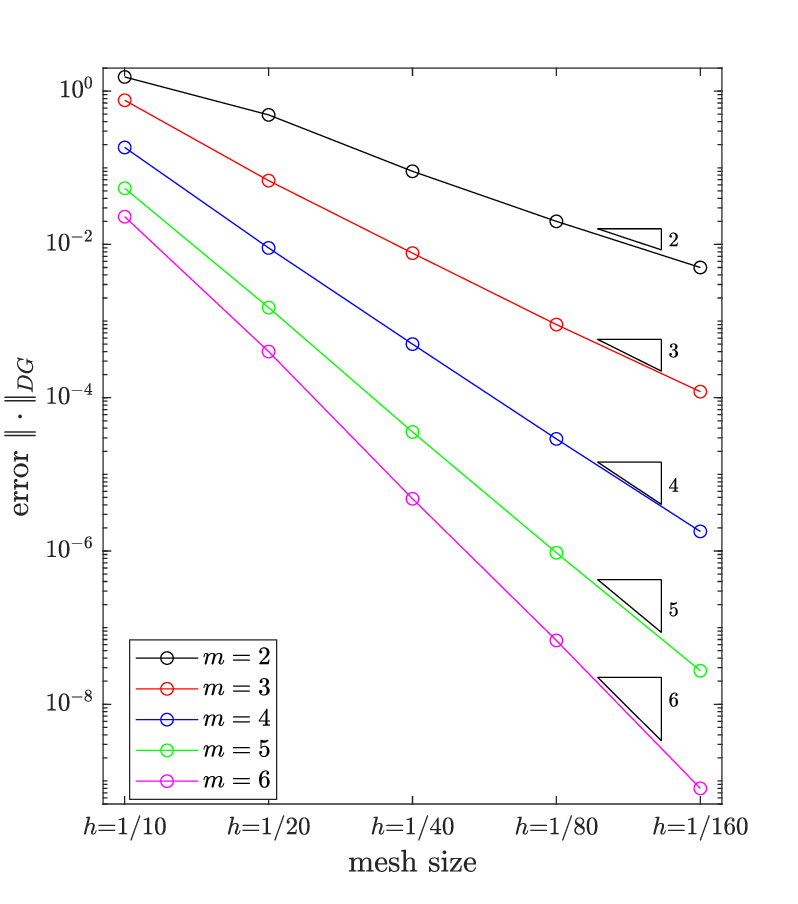}
  \hspace{30pt}
  \includegraphics[width=0.45\textwidth]{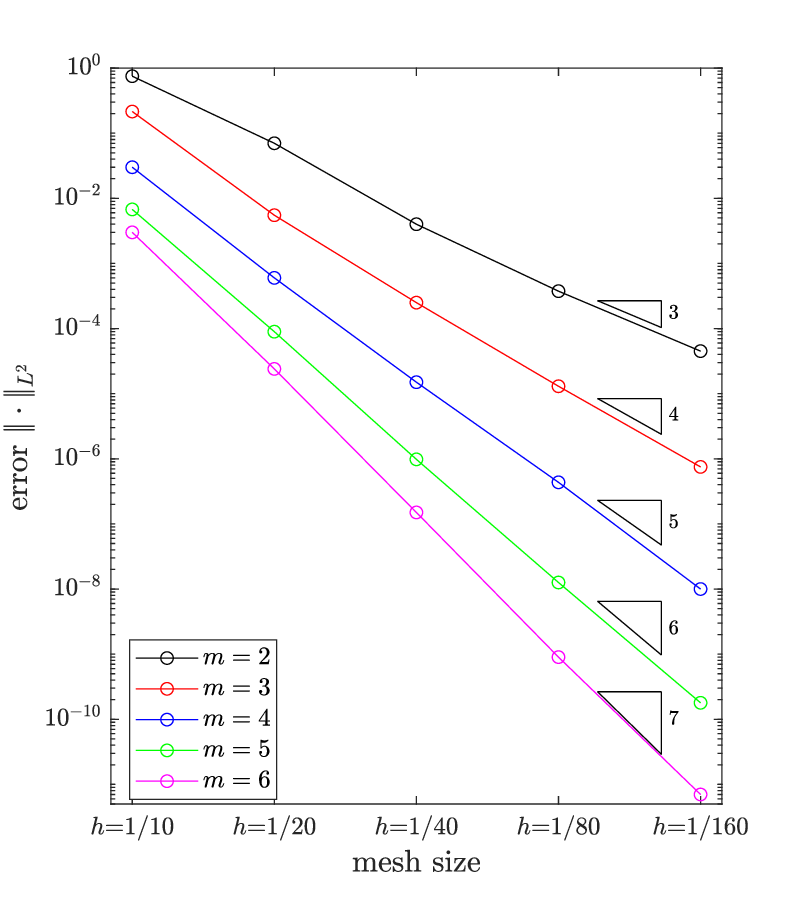}
    \caption{2D accuracy test, $k=10$.}
    \label{ch5_2}
\end{figure}

\begin{figure}[htbp]
  \centering
  \includegraphics[width=0.45\textwidth]{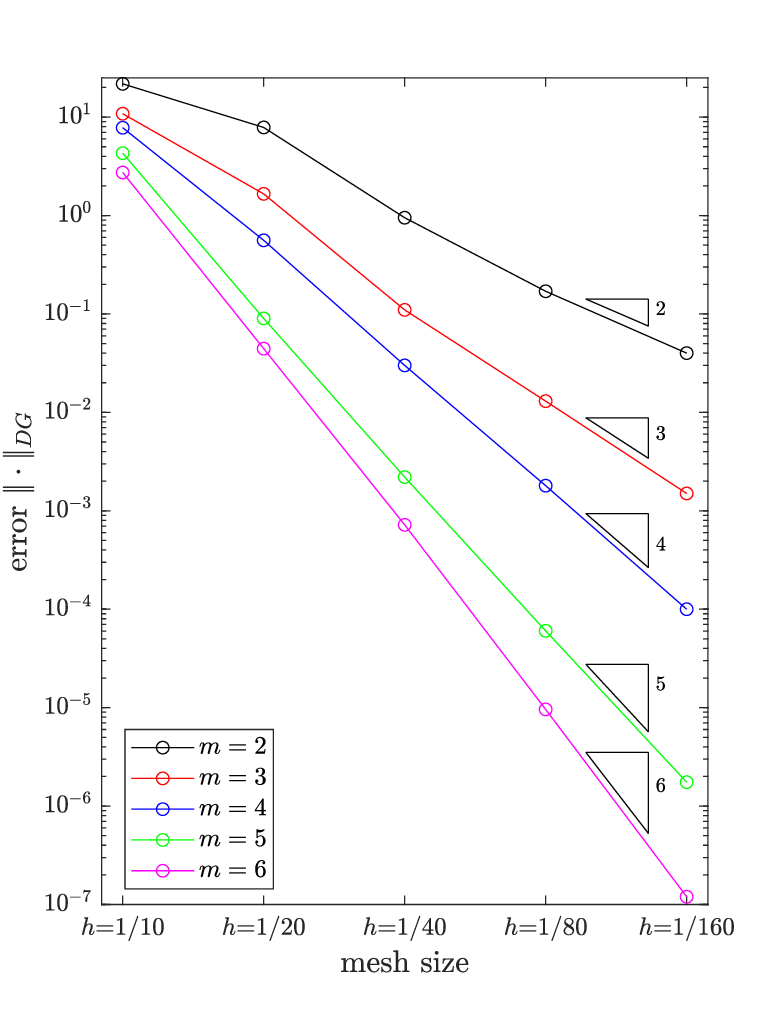}
  \hspace{30pt}
  \includegraphics[width=0.46\textwidth]{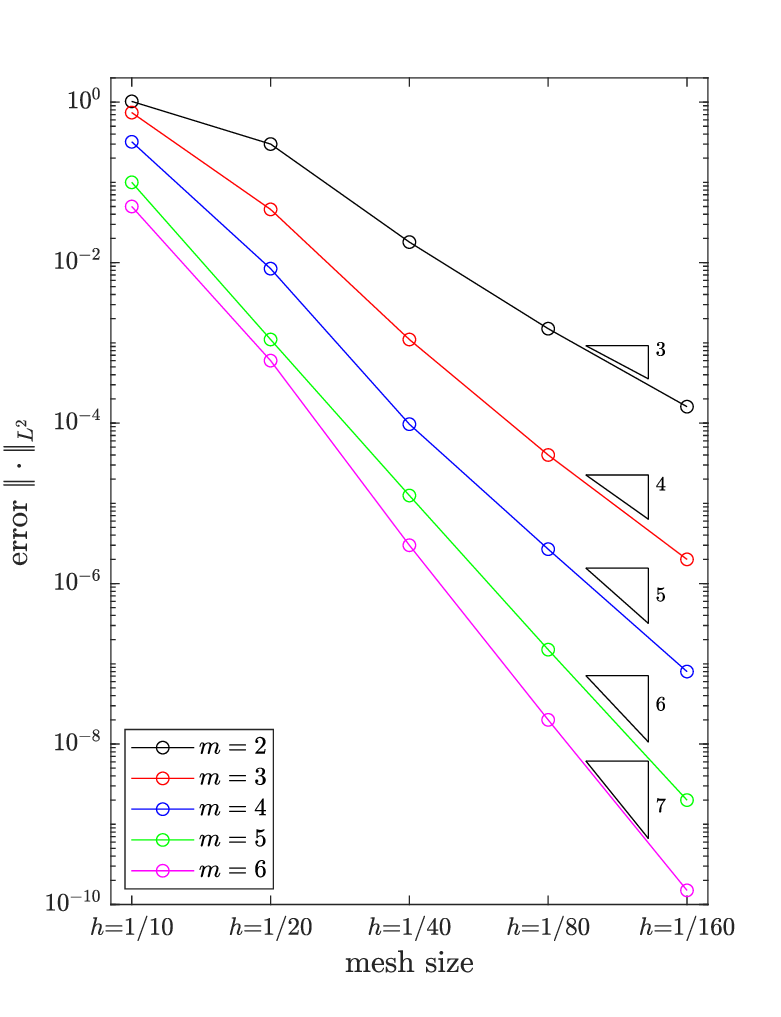}
    \caption{2D accuracy test, $k=20$.}
      \label{ch5_3}
\end{figure}

\noindent \textbf{Example 2}
We conduct a comprehensive comparative study between our proposed method
and the conventional DG formulation for the Helmholtz problem.
Following \cite{hughes2000comparison}, the number of degrees of freedom
in the discretized system serves as an appropriate metric for 
evaluating numerical efficiency. Our analysis examines the relative 
performance of both methodologies across a sequence of refined meshes 
for the case $k=20$, with polynomial orders ranging from $2 \le m \le 6$.
Table~\ref{tab_ex1_error} presents the relative $L^2$ error ratios
between the RDA and DG approaches when employing identical \# DOF, 
clearly demonstrating the superior efficiency of the RDA method.
Figure~\ref{effi_compare} illustrates the $L^2$ error as a function of
both the \# DOF and the number of non-zero matrix entries. 
These results indicate that our RDA scheme achieves enhanced 
computational efficiency relative to the \# DOF compared to standard
DG implementations. Furthermore, for approximation orders $m \ge 3$, 
the RDA method requires fewer non-zero matrix elements than the DG 
counterpart to attain comparable $L^2$ accuracy. A more quantitative 
analysis is provided in Table~\ref{tab_ex1_dofs}, which details the 
relative ratios of \# DOF and non-zero matrix entries required by the 
RDA method to achieve $L^2$ errors comparable to the DG method.

\begin{figure}[htbp]
  \centering
  \includegraphics[width=0.45\textwidth]{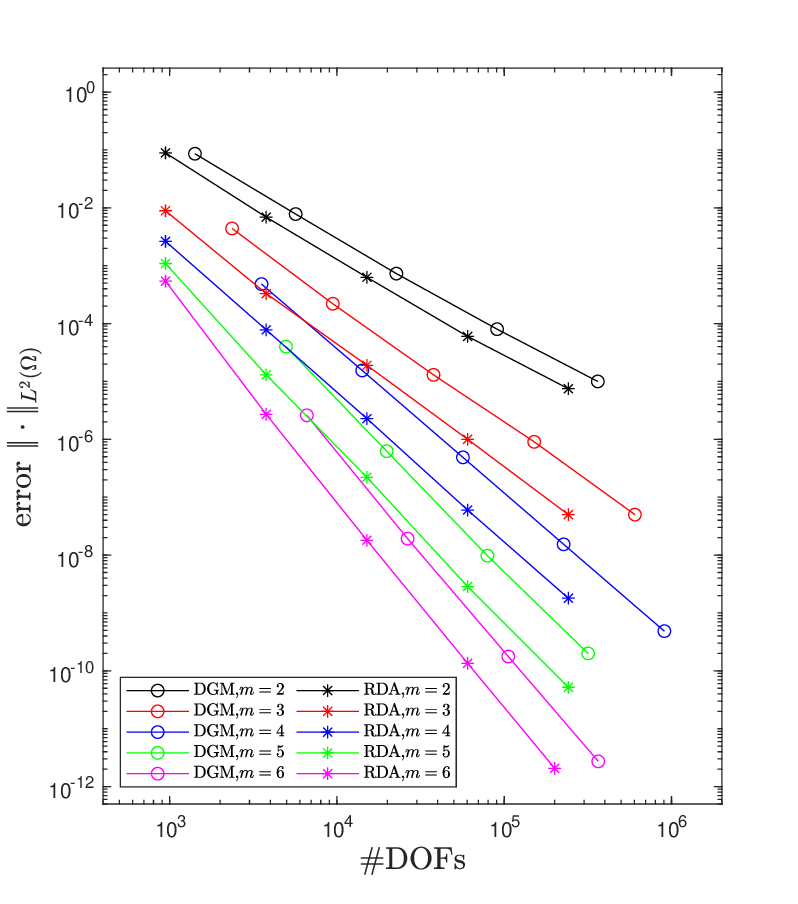}
  \hspace{30pt}
  \includegraphics[width=0.45\textwidth]{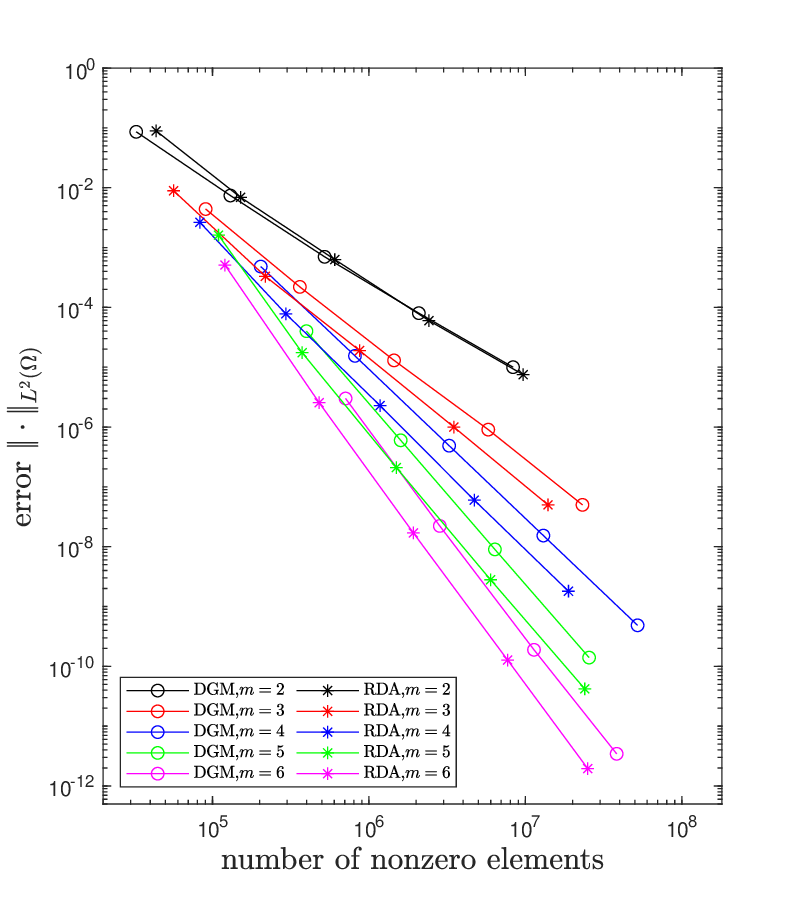}
  \caption{Computational efficiency comparison in two dimensions.}
   \label{effi_compare}
\end{figure}

\begin{table}[htbp]
  \centering
  \renewcommand{\arraystretch}{1.3}
  \begin{tabular}{l|c|c|c|c|c}
    \hline\hline
    $m$  & 2 & 3 & 4 & 5 & 6 \\ 
    \hline
    Relative $L^2$ error (RDA/DG) & 0.527 & 0.303 & 0.167 & 0.105 & 0.070  \\ 
    \hline\hline
  \end{tabular}
  \caption{Relative $L^2$ error of the RDA method compared to the DG method using identical \# DOF (Example 2).}
  \label{tab_ex1_error}
\end{table}

\begin{table}[htbp]
  \centering
  \renewcommand{\arraystretch}{1.3}
  \begin{tabular}{l|c|c|c|c|c}
    \hline\hline
    $m$  & 2 & 3 & 4 & 5 & 6 \\ 
    \hline
    \# DOF ratio (RDA/DG) & 65.6\% & 48.1\% & 38.9\% & 36.2\% & 34.3\% \\ 
    \hline
    Non-zero entries ratio (RDA/DG) & 95.5\% & 80.8\% & 70.2\% & 63.6\% & 59.8\% \\ 
    \hline\hline
  \end{tabular}
  \caption{Computational cost comparison for achieving comparable $L^2$ accuracy (Example 2).}
  \label{tab_ex1_dofs}
\end{table}

\par \textbf{Example 3}
Following the methodology outlined in the preceding section,
we implement the PGMRES iterative solver with $P^{-1}$
serving as the preconditioner.
Numerical experiments are conducted for both 
parameter configurations $\epsilon=0$ and $\epsilon=k^2$.
The solution of the auxiliary system $P\mathbf{y} = \mathbf{z}$
within each Krylov subspace iteration is computed using the 
geometric multigrid solver described previously.
Iteration counts required for convergence across varying 
wavenumbers $k$ are compiled in Tables~\ref{ch5_01}, \ref{ch5_02}, and \ref{ch5_03}.
The results demonstrate that employing $P^{-1}$ as a preconditioner
yields GMRES iteration counts that appear to approach a uniform upper
bound as the mesh parameter $h$ approaches zero.

\begin{table}[htbp]
    \centering
    \renewcommand{\arraystretch}{1.3}
    \begin{tabular}{c|c|c|c|c|c}
      \hline\hline
        \diagbox[width=1.75cm]{$m$}{$1/h$}  & 10 & 20 & 40 & 80 & 160 \\
      \hline
       2 & 48/45 & 46/42 & 45/44 & 44/40 & 41/39  \\
      \hline
       3 & 55/52 & 47/47 & 46/45 & 47/44 & 49/46 \\
      \hline
       4 & 50/49 & 60/58 & 61/59 & 60/56 & 56/52 \\
      \hline
       5 & 56/55 & 74/72 & 82/80 & 83/80 & 83/79 \\
      \hline
       6 & 81/79 & 88/84 & 100/95 & 114/105 & 125/118 \\       
      \hline\hline
    \end{tabular}
    \caption{Iteration counts for two-dimensional problem with $k=5$ .}
    \label{ch5_01}
\end{table}

\begin{table}[htbp]
    \centering
    \renewcommand{\arraystretch}{1.3}
    \begin{tabular}{c|c|c|c|c|c}
      \hline\hline
        \diagbox[width=1.75cm]{$m$}{$1/h$} & 10 & 20 & 40 & 80 & 160\\
      \hline
       2 & 67/58 & 76/62 & 77/65 & 77/59 & 75/57 \\
      \hline
       3 & 78/66 & 77/70 & 77/68 & 83/67 & 84/69 \\
      \hline
       4 & 74/68 & 89/79 & 99/87 & 100/85 & 100/82 \\
      \hline
       5 & 79/72 & 102/92 & 118/105 & 134/113 & 137/115 \\
      \hline
       6 & 103/92 & 122/109 & 142/130 & 221/205 & 262/231 \\    
      \hline\hline
    \end{tabular}
    \caption{Iteration counts for two-dimensional problem with $k=10$ .}
    \label{ch5_02}
\end{table}

\begin{table}[htbp]
    \centering
    \renewcommand{\arraystretch}{1.3}
    \begin{tabular}{c|c|c|c|c|c}
      \hline\hline
        \diagbox[width=1.75cm]{$m$}{$1/h$}  & 10 & 20 & 40 & 80 & 160 \\
      \hline
       2 & 115/75 & 178/83 & 220/92 & 256/104 & 264/109 \\
      \hline
       3 & 180/82 & 205/89 & 216/95 & 266/110 & 273/126 \\
      \hline
       4 & 174/84 & 216/101 & 275/124 & 343/142 & 357/150  \\
      \hline
       5 & 178/96 & 248/125 & 340/138 & 446/180 & 506/190  \\
      \hline
       6 & 148/128 & 294/145 & 443/212 & 681/262 & 874/297  \\
      \hline\hline
    \end{tabular}
    \caption{Iteration counts for two-dimensional problem with $k=20$ .}
    \label{ch5_03}
\end{table}
To further demonstrate the computational efficiency of the proposed 
$P^{-1}$+GMG preconditioner, in Table \ref{RDA_time} we present execution 
times for the GMRES method applied to the Helmholtz problem with $k = 20$.
Compared to the widely used BoomerAMG preconditioner from Hypre, our
$P^{-1}$+GMG method demonstrates superior performance in terms of 
both iteration counts and computational time, particularly as the 
mesh is refined and for higher values of $m$. The results
consistently show that $P^{-1}$+GMG requires fewer iterations and 
significantly less CPU time across all tested configurations, with
the performance gap widening as $h$ decreases. 
These findings highlight the effectiveness of our preconditioner in handling
Helmholtz problems with discontinuous Galerkin discretizations.
The time cost for the solver of traditional discontinuous Galerkin method
is also provided here for reference in Table \ref{DG_time}.

\begin{table}[htbp]
    \centering
    \renewcommand{\arraystretch}{1.5}
    \begin{tabular}{c|c|c|c|c|c|c}
        \hline\hline
         $m$ &  \diagbox[width=2.5cm]{$m$}{$1/h$} & 10 & 20 & 40 & 80 & 160 \\
        \cline{3-7}
        \hline
        \multirow{2}{*}{2} & $P^{-1}$+GMG & 0.030 (115) & 0.115 (178) & 0.801 (220) & 3.783 (256) & 19.617 (264) \\
        \cline{2-7}
        & BoomerAMG & 0.024 (67) & 0.265 (170) & 1.945 (293) & 13.119 (431) & 86.027 (606) \\
        \hline
        \multirow{2}{*}{3} & $P^{-1}$+GMG & 0.035 (180) & 0.162 (205) & 0.944 (216) & 4.598 (266) & 22.951 (273) \\
        \cline{2-7}
        & BoomerAMG & 0.049 (91) & 0.518 (207) & 3.412 (320) & 22.568 (463) & 141.721 (620) \\
        \hline
        \multirow{2}{*}{4} & $P^{-1}$+GMG & 0.040 (174) & 0.266 (216) & 1.513 (275) & 7.987 (343) & 39.311 (357) \\
        \cline{2-7}
        & BoomerAMG & 0.079 (118) & 0.785 (208) & 5.296 (387) & 39.943 (566) & 236.073 (680) \\
        \hline
        \multirow{2}{*}{5} & $P^{-1}$+GMG & 0.072 (178) & 0.343 (248) & 1.925 (340) & 11.802 (446) & 57.160 (506) \\
        \cline{2-7}
        & BoomerAMG & 0.114 (122) & 1.215 (263) & 7.946 (388) & 58.290 (585) & 298.676 (691) \\
        \hline
        \multirow{2}{*}{6} & $P^{-1}$+GMG & 0.128 (148) & 0.700 (294) & 2.947 (443) & 21.437 (681) & 119.969 (874) \\
        \cline{2-7}
        & BoomerAMG & 0.194 (190) & 1.993 (389) & 15.980 (610) & 166.05 (1085) & >300  \\
        \hline\hline
    \end{tabular}
    \caption{CPU time (seconds) and iteration counts (in parentheses) for GMRES with different preconditioners.}
       \label{RDA_time}
\end{table}
\begin{table}[htbp]
    \centering
    \renewcommand{\arraystretch}{1.5}
    \begin{tabular}{c|c|c|c|c}
      \hline\hline
        \diagbox[width=1.75cm]{$m$}{$1/h$}  & 10 & 20 & 40 & 80  \\
      \hline
       2 & 0.41 (176) & 2.62 (210) & 13.90 (282) & 93.44 (383)  \\
      \hline
       3 & 1.08 (157) & 5.52 (237) & 35.5 (313) & 195.00 (420) \\
      \hline
       4 & 2.23 (198) & 13.41 (272) & 79.46 (350) & >300   \\
      \hline
       5 & 5.52 (216) & 36.74 (269) & 196.03 (416) & >300    \\
      \hline
       6 & 12.66 (368) & 84.45 (456) & >300  & > 300    \\    
      \hline\hline
    \end{tabular}
    \caption{CPU time cost for DG method with BoomerAMG preconditioner.}
    \label{DG_time}
  \end{table}

\par \textbf{Example 4}
This example demonstrates the capability of our method to handle
large wavenumber problems. We solve
the pure Helmholtz problem with the analytical solution given by
    \begin{displaymath}
  u(x,y) = \frac{\cos(kr)}{k} - \frac{\cos k + i \sin k}{k (J_0(k) + i 
  J_1(k))} J_0(kr),\quad r = \sqrt{(x-0.5)^2+(y-0.5)^2},
\end{displaymath}   
on the unit square domain $[0, 1]^2$ with wavenumber $k=100$.
Numerical simulations are performed on an unstructured mesh 
with $h=1/640$ using reconstruction order $m=2$. Under these conditions,
the PGMRES algorithm converges in 2,936 iterations with a total 
execution time of 5,808 seconds. 
Figure~\ref{sol} provides a comparative visualization of 
the numerical solution alongside the exact solution.

\begin{figure}[htbp]
  \centering
  \includegraphics[width=0.45\textwidth]{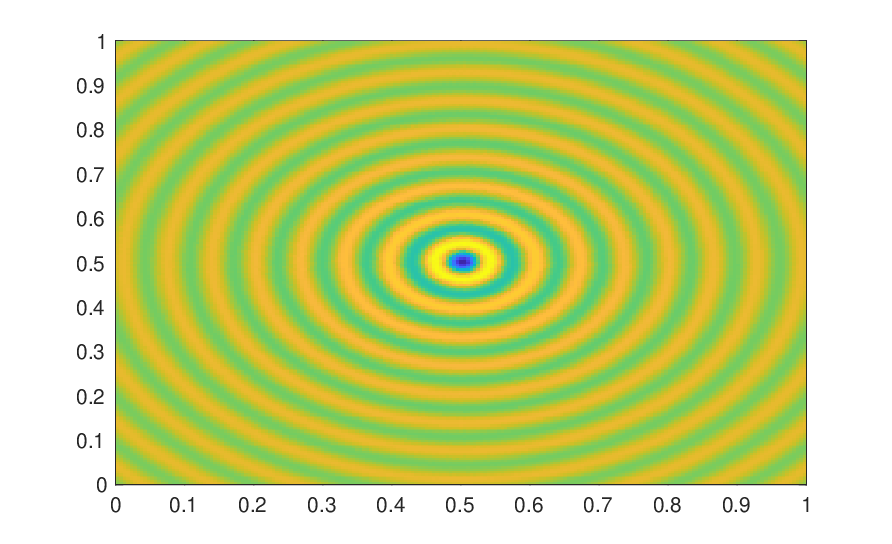}
  \hspace{30pt}
  \includegraphics[width=0.45\textwidth]{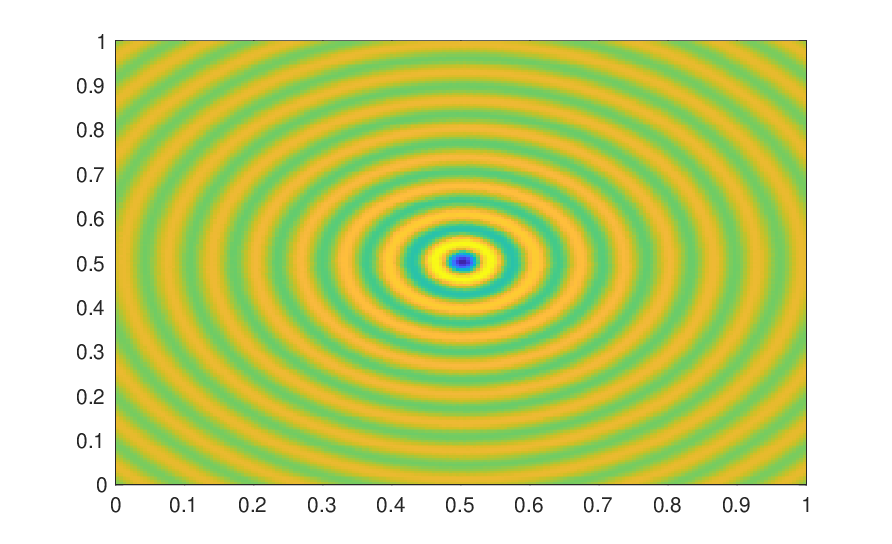}
  \caption{Numerical solution (left) and exact solution (right).}
   \label{sol}
\end{figure}

\noindent \textbf{Example 5}
In this final example, we solve a three-dimensional Helmholtz problem
defined on the cubic domain $\Omega = (0,1)^3$, with the analytical solution specified as
\begin{displaymath}
u(x,y,z) = e^{i k (x \sin \theta \cos \phi + y \sin \theta \sin \phi + z \cos \theta)},
\end{displaymath}
where the directional parameters are set to $\theta = \frac{\pi}{4}$
and $\phi = \frac{\pi}{5}$. Numerical tests are conducted on 
sequences of quasi-uniform tetrahedral meshes with $h=1/8, 1/16, 1/32, 1/64$
for wavenumber $k=8$. Figure~\ref{ch5_4} displays the convergence 
behavior under both $L^2$ and energy norms, 
confirming the theoretical convergence rates.

Following the methodology of Example 2, we extend our analysis to
include a comprehensive efficiency comparison in three dimensions.
This comparison examines the trade-off between computational cost
and numerical accuracy, with detailed results summarized in 
Table~\ref{tab_ex5_error} and Table~\ref{tab_ex5_dofs}. 
The comparative data is graphically illustrated in Figure~\ref{effi_compare_2},
demonstrating that the RDA method requires significantly fewer 
\# DOF and non-zero matrix elements than the traditional DG approach,
particularly for higher approximation orders.

The performance of the preconditioner $P^{-1}$ in three dimensions
is documented in Table~\ref{ch5_04}. Notably, the iteration counts 
remain uniformly bounded with respect to mesh refinement for both 
parameter configurations ($\epsilon = 0$ and $\epsilon = k^2$),
underscoring the robustness and scalability of our preconditioning strategy.

\begin{figure}[htbp]
  \centering
  \includegraphics[width=0.45\textwidth]{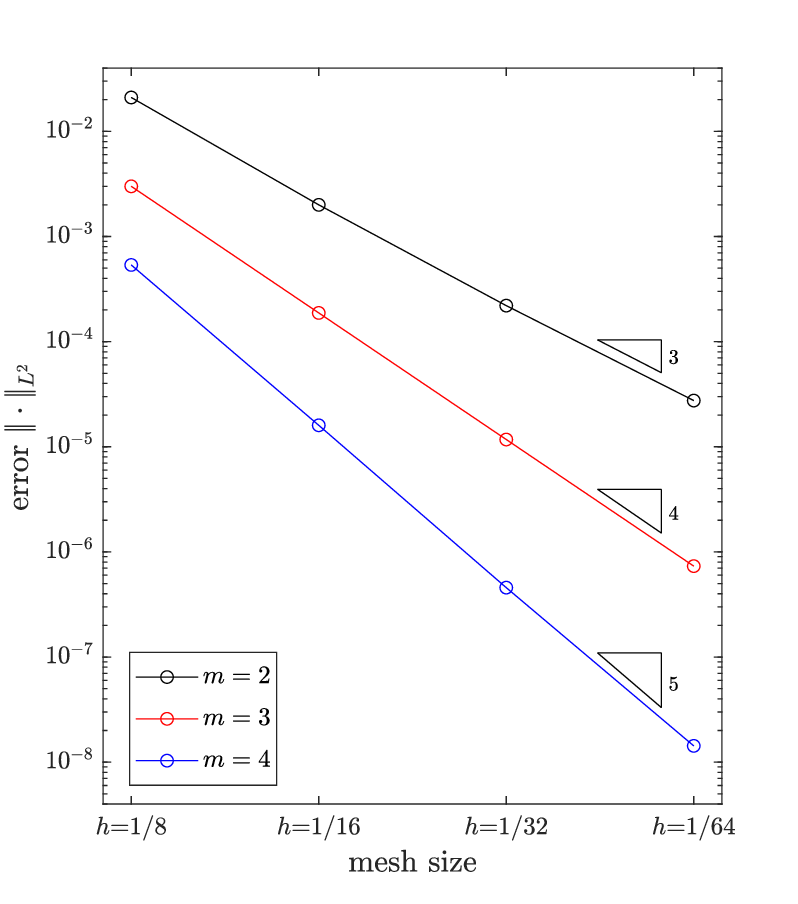}
  \hspace{30pt}
  \includegraphics[width=0.45\textwidth]{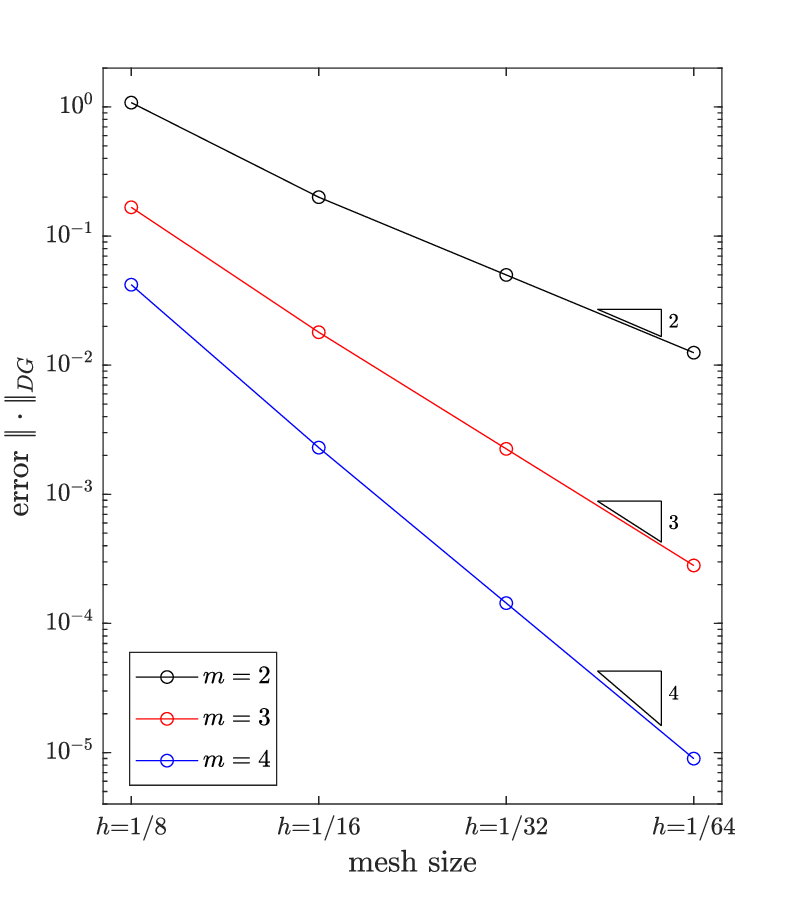}
    \caption{Three-dimensional accuracy test, $k=8$.}
      \label{ch5_4}
\end{figure}

\begin{figure}[htbp]
  \centering
  \includegraphics[width=0.45\textwidth]{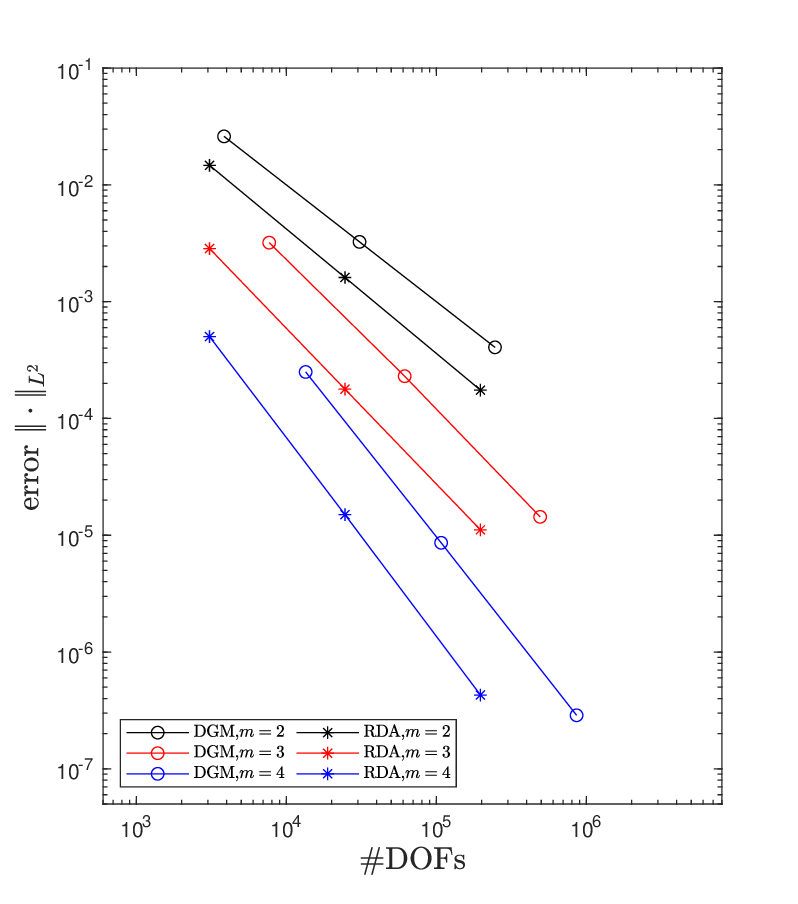}
  \hspace{30pt}
  \includegraphics[width=0.45\textwidth]{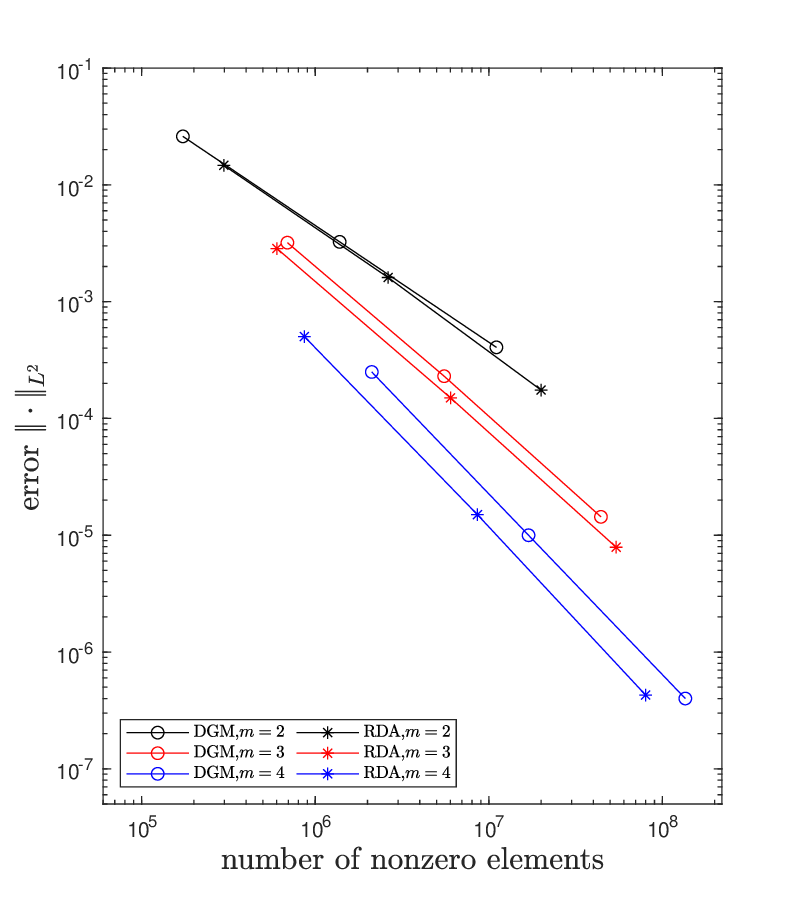}
  \caption{Computational efficiency comparison in three dimensions.}
   \label{effi_compare_2}
\end{figure}

\begin{table}[htbp]
  \centering
  \renewcommand{\arraystretch}{1.3}
  \begin{tabular}{l|c|c|c}
    \hline\hline
    $m$  & 2 & 3 & 4 \\ 
    \hline
    Relative $L^2$ error (RDA/DG) & 0.501 & 0.245 & 0.084 \\ 
    \hline\hline
  \end{tabular}
  \caption{Relative $L^2$ error of the RDA method compared to the DG method using identical \# DOF (Example 5).}
  \label{tab_ex5_error}
\end{table}

\begin{table}[htbp]
  \centering
  \renewcommand{\arraystretch}{1.5}
  \begin{tabular}{l|c|c|c}
    \hline\hline
    $m$  & 2 & 3 & 4  \\ 
    \hline
    \# DOF ratio (RDA/DG) & 49.8\% & 32.6\% & 23.2\% \\ 
    \hline
    Non-zero entries ratio (RDA/DG) & 83.2\% & 68.1\% & 44.6\% \\ 
    \hline\hline
  \end{tabular}
  \caption{Computational cost comparison for achieving comparable $L^2$ accuracy in three dimensions (Example 5).}
  \label{tab_ex5_dofs}
\end{table}

\begin{table}[htbp]
    \centering
    \renewcommand{\arraystretch}{1.5}
  \begin{tabular}{l|c|c|c|c}
      \hline\hline
        \diagbox[width=1.85cm]{$m$}{$1/h$} & 8 & 16 & 32 & 64 \\
      \hline
       2 & 135/116 & 144/135 & 147/139 & 155/142 \\
      \hline
       3 & 128/105 & 148/112 &  164/153 & 177/160 \\
      \hline
       4 & 125/132 & 172/165 &  248/194 & 279/219 \\
      \hline\hline
    \end{tabular}
    \caption{Iteration counts for three-dimensional problem with $k=8$ .}
    \label{ch5_04}
  \end{table}

\section{Conclusion}
\label{sec_conclusion}

In this paper, we have developed and analyzed an efficient solver
for the Helmholtz equation based on a novel approximation space.
The proposed method demonstrates superior performance over traditional approaches,
achieving higher accuracy with reduced degrees of freedom 
and lower memory requirements. 
By incorporating a natural preconditioner through piecewise constant 
discretization, the method also significantly reduces computational time.
Numerical experiments in both 2D and 3D confirm its effectiveness.

\section*{Acknowledgements}
The authors would like to thank Professor Ruo Li from Peking University and Fanyi Yang from Sichuan University for their valuable suggestions and advice, which have significantly improved the quality of this manuscript. This research was supported by the High-performance Computing Platform of Peking University.

\bibliographystyle{amsplain}
\bibliography{../ref}

\providecommand{\bysame}{\leavevmode\hbox to3em{\hrulefill}\thinspace}
\providecommand{\MR}{\relax\ifhmode\unskip\space\fi MR }
\providecommand{\MRhref}[2]{%
  \href{http://www.ams.org/mathscinet-getitem?mr=#1}{#2}
}
\providecommand{\href}[2]{#2}
\begin{thebibliography}{10}

\bibitem{Congreve2019robust}
S.~Congreve, J.~Gedicke, and I.~Perugia, \emph{Robust adaptive {$hp$}
  discontinuous {G}alerkin finite element methods for the {H}elmholtz
  equation}, SIAM J. Sci. Comput. \textbf{41} (2019), no.~2, A1121--A1147.
  \MR{3937921}

\bibitem{haijun2020Helm}
Songyao Duan and Haijun Wu, \emph{Adaptive fem for helmholtz equation with
  large wavenumber}, Journal of Scientific Computing \textbf{94} (2022).

\bibitem{Erlangga2004}
Y.~A. Erlangga, C.~Vuik, and C.~W. Oosterlee, \emph{On a class of
  preconditioners for solving the helmholtz equation}, Applied Numerical
  Mathematics \textbf{50} (2004), no.~3, 409--425.

\bibitem{Erlangga2006}
\bysame, \emph{A novel multigrid based preconditioner for heterogeneous
  helmholtz problems}, SIAM Journal on Scientific Computing \textbf{27} (2006),
  1471--1492.

\bibitem{Farhat2003discontinuous}
C.~Farhat, I.~Harari, and U.~Hetmaniuk, \emph{A discontinuous {G}alerkin method
  with {L}agrange multipliers for the solution of {H}elmholtz problems in the
  mid-frequency regime}, Comput. Methods Appl. Mech. Engrg. \textbf{192}
  (2003), no.~11-12, 1389--1419.

\bibitem{Feng2009discontinuous}
X.~Feng and H.~Wu, \emph{Discontinuous {G}alerkin methods for the {H}elmholtz
  equation with large wave number}, SIAM J. Numer. Anal. \textbf{47} (2009),
  no.~4, 2872--2896.

\bibitem{Feng2011discontinuous}
\bysame, \emph{{$hp$}-discontinuous {G}alerkin methods for the {H}elmholtz
  equation with large wave number}, Math. Comp. \textbf{80} (2011), no.~276,
  1997--2024.

\bibitem{gmreshelm}
M.~Gander, I.~Graham, and E.~Spence, \emph{Applying gmres to the helmholtz
  equation with shifted laplacian preconditioning: what is the largest shift
  for which wavenumber-independent convergence is guaranteed?}, Numerische
  Mathematik \textbf{131} (2015), 567--614.

\bibitem{ddm2021helm}
S.~Gong, I.~Graham, and E.~Spence, \emph{Domain decomposition preconditioners
  for high-order discretizations of the heterogeneous helmholtz equation}, IMA
  Journal of Numerical Analysis \textbf{41} (2021), 2139--2185.

\bibitem{ddm2023helm}
\bysame, \emph{Convergence of restricted additive schwarz with impedance
  transmission conditions for discretised helmholtz problems}, Mathematics of
  Computation \textbf{92} (2023), 175--215.

\bibitem{ddm2015helm}
I.~Graham, E.~Spence, and E.~Vainikko, \emph{Domain decomposition
  preconditioning for high-frequency helmholtz problems with absorption},
  Mathematics of Computation \textbf{86} (2015), 2089--2127.

\bibitem{ddm2020helm}
I.~Graham, E.~Spence, and J.~Zou, \emph{Domain decomposition with local
  impedance conditions for the helmholtz equation with absorption}, SIAM
  Journal on Numerical Analysis \textbf{58} (2020), 2515--2543.

\bibitem{Hoppe2013convergence}
R.~H.~W. Hoppe and N.~Sharma, \emph{Convergence analysis of an adaptive
  interior penalty discontinuous {G}alerkin method for the {H}elmholtz
  equation}, IMA J. Numer. Anal. \textbf{33} (2013), no.~3, 898--921.
  \MR{3081488}

\bibitem{Hu2020novel}
Q.~Hu and R.~Song, \emph{A novel least squares method for {H}elmholtz equations
  with large wave numbers}, SIAM J. Numer. Anal. \textbf{58} (2020), no.~5,
  3091--3123.

\bibitem{hughes2000comparison}
T.~J.~R. Hughes, G.~Engel, L.~Mazzei, and M.~G. Larson, \emph{A comparison of
  discontinuous and continuous {G}alerkin methods based on error estimates,
  conservation, robustness and efficiency}, Discontinuous {G}alerkin methods
  ({N}ewport, {RI}, 1999), Lect. Notes Comput. Sci. Eng., vol.~11, Springer,
  Berlin, 2000, pp.~135--146.

\bibitem{Ihlenburg1995finite}
F.~Ihlenburg and I.~Babu\v{s}ka, \emph{Finite element solution of the
  {H}elmholtz equation with high wave number. {I}. {T}he {$h$}-version of the
  {FEM}}, Comput. Math. Appl. \textbf{30} (1995), no.~9, 9--37.

\bibitem{Ihlenburg1997finite}
\bysame, \emph{Finite element solution of the {H}elmholtz equation with high
  wave number. {II}. {T}he {$h$}-{$p$} version of the {FEM}}, SIAM J. Numer.
  Anal. \textbf{34} (1997), no.~1, 315--358.

\bibitem{Li2023preconditioned}
R.~Li, Q.~Liu, and F.~Yang, \emph{Preconditioned nonsymmetric/symmetric
  discontinuous {G}alerkin method for elliptic problem with reconstructed
  discontinuous approximation}, accpeted by J. Sci. Comput. (2023).

\bibitem{Li2016discontinuous}
R.~Li, P.~Ming, Z.~Sun, and Z.~Yang, \emph{An arbitrary-order discontinuous
  {G}alerkin method with one unknown per element}, J. Sci. Comput. \textbf{80}
  (2019), no.~1, 268--288.

\bibitem{Li2012efficient}
R.~Li, P.~Ming, and F.~Tang, \emph{An efficient high order heterogeneous
  multiscale method for elliptic problems}, Multiscale Model. Simul.
  \textbf{10} (2012), no.~1, 259--283.

\bibitem{Li2019reconstructed}
R.~Li and F.~Yang, \emph{A reconstructed discontinuous approximation to
  {M}onge-{A}mp\`ere equation in least square formulation}, Adv. Appl. Math.
  Mech. \textbf{15} (2023), no.~5, 1109--1141. \MR{4613677}

\bibitem{xuejun}
P.~Lu, X.~Xu, B.~Zheng, and J.~Zou, \emph{Two-level hybrid schwarz
  preconditioners for the helmholtz equation with high wave number}, submitted
  (2024).

\bibitem{Melenk2011wavenumber}
J.~M. Melenk and S.~Sauter, \emph{Wavenumber explicit convergence analysis for
  {G}alerkin discretizations of the {H}elmholtz equation}, SIAM J. Numer. Anal.
  \textbf{49} (2011), no.~3, 1210--1243.

\bibitem{Nguyen2015hybridizable}
N.~C. Nguyen, J.~Peraire, F.~Reitich, and B.~Cockburn, \emph{A phase-based
  hybridizable discontinuous {G}alerkin method for the numerical solution of
  the {H}elmholtz equation}, J. Comput. Phys. \textbf{290} (2015), 318--335.

\bibitem{Thompson1995Galerkin}
L.~L. Thompson and P.~M. Pinsky, \emph{A {G}alerkin least-squares finite
  element method for the two-dimensional {H}elmholtz equation}, Internat. J.
  Numer. Methods Engrg. \textbf{38} (1995), no.~3, 371--397.

\end{thebibliography}

\end{document}